\documentclass[a4paper, 11pt]{article}

\usepackage[utf8]{inputenc}			
\usepackage[T1]{fontenc}				
\usepackage[english]{babel}

\usepackage{amsmath}
\usepackage{amssymb}
\usepackage{amsthm}
	\theoremstyle{plain}
		\newtheorem*{mainthm}{Theorem}		
		\newtheorem*{maincor}{Corollary}
		\newtheorem{thm}{Theorem}[section]	
		\newtheorem{cor}[thm]{Corollary}		
		\newtheorem{lem}[thm]{Lemma}		
		\newtheorem{prop}[thm]{Proposition}

	\theoremstyle{definition}
		\newtheorem{defn}[thm]{Definition}	

	\theoremstyle{remark}
		\newtheorem{rmk}[thm]{Remark}		

\numberwithin{equation}{section}	

\usepackage{mathtools}		
\usepackage{mathrsfs}		
\usepackage{eucal}			

\usepackage{braket}			
\usepackage[all,pdf]{xy}		

\usepackage{mparhack}		
\usepackage{fixltx2e}		
\usepackage{relsize}			

\usepackage{a4wide}		

\usepackage{booktabs}		
\usepackage{multirow}		
\usepackage{caption}		
\usepackage{varioref}		

\usepackage{enumerate}		
\usepackage[colorlinks=true, linkcolor=black, citecolor=black, urlcolor=blue]{hyperref}		

\newcommand{\A}{\mathscr{A}}	
\newcommand{\B}{\mathscr{B}}		
\newcommand{\C}{\mathbb{C}}		
\newcommand{\Cscr}{\mathscr{C}}
\newcommand{\D}{\mathcal{D}}
\newcommand{\G}{\mathscr{G}}
\renewcommand{\H}{\mathcal{H}}	
\newcommand{\I}{\mathcal{I}}		
\newcommand{\M}{\mathcal{M}}	
\newcommand{\N}{\mathbb{N}}		
\newcommand{\R}{\mathbb{R}}		
\renewcommand{\S}{\mathcal{S}}	
\newcommand{\U}{\mathbb{U}}		
\newcommand{\W}{\mathcal{W}}
\newcommand{\Z}{\mathbb{Z}}		

\newcommand{\Bsa}{\B^{\textup{sa}}}

\newcommand{\SO}{\mathrm{SO}}
\newcommand{\Sp}{\mathrm{Sp}}
\newcommand{\ssp}{\mathfrak{sp}}
\newcommand{\Mat}{\mathrm{Mat}}

\let\d\relax
\newcommand{\d}{\mathrm{d}}				
\newcommand{\eps}{\varepsilon}
\newcommand{\norm}[1]{\left\lVert #1 \right\rVert}			
\newcommand{\opnorm}[2]{\norm{#1}_{\mathscr{L}(#2)}}		
\newcommand{\abs}[1]{\left\lvert #1 \right\rvert}				
\newcommand{\trasp}[1]{{#1}^\mathsf{T}}					
\newcommand{\Uhom}{U_{\alpha}}
\newcommand{\Ulog}{U_{\log}}
\newcommand{\nmext}{n^-_{\textup{ext}}}		
\newcommand{\npext}{n^+_{\textup{ext}}}		
\newcommand{\iMor}{i_{\textup{Morse}}}		

\DeclareMathOperator{\diag}{diag}		
\DeclareMathOperator{\spfl}{sf}			
\DeclareMathOperator{\sgn}{sgn}		
\DeclareMathOperator{\tr}{tr}			
\DeclareMathOperator{\ord}{ord}		
\DeclareMathOperator{\gen}{span}
\DeclareMathOperator{\im}{im}			

\renewcommand{\leq}{\leqslant}
\renewcommand{\geq}{\geqslant}
\renewcommand{\hat}{\widehat}
\renewcommand{\tilde}{\widetilde}
\renewcommand{\=}{\coloneqq}			
\newcommand{\eq}{\eqqcolon}			

\newcommand{\ie}{i.e.~}

\newcommand{\email}[1]{\href{mailto:#1}{\textsf{#1}}}

\title{Linear instability of relative equilibria\\ for $n$-body problems in the plane}
\author{Vivina L.~Barutello, Riccardo D.~Jadanza, Alessandro Portaluri%
	\thanks{The authors were partially supported by PRIN 2009 ``Critical Point Theory and Perturbative Methods for Non-Linear Differential Equations''.}}
\date{\today}

\begin{document}

	\maketitle

	\begin{abstract}
		Following Smale, we study simple symmetric mechanical systems of $n$ point particles in the plane. In particular, we address the question of the linear and spectral 
		stability properties of relative equilibria, which are special solutions of the equations of motion. 

		Our main result is a sufficient condition to detect spectral (hence linear) instability. Namely, we prove that if the Morse index of an equilibrium point with even nullity is odd, 
		then the associated relative equilibrium is spectrally unstable. The proof is based on some refined formul\ae\  for computing the spectral flow.
		
		As a notable application of our theorem, we examine two important classes of singular potentials: the $\alpha$-homogeneous one, with $\alpha \in (0, 2)$, 
		which includes the gravitational case, and the logarithmic one. We also establish, for the $\alpha$-homogeneous potential,
		an inequality  which is useful to test the spectral instability of the associated relative equilibrium. 
		\vspace{0.5truecm}

		\noindent
		\emph{MSC Subject Class}: Primary 70F10; Secondary 37C80.
		
		\vspace{0.5truecm}
		
	\end{abstract}
	

	\section{Introduction}

	Simple mechanical systems are a special class of Hamiltonian systems in which the Hamiltonian function can be written as the sum of the potential and kinetic energies.
	The search for special orbits, such as equilibria and periodic orbits, and the understanding of their stability properties are amongst the major subjects in the whole theory of Dynamical
	Systems.
	
	In 1970, in one of his famous papers \cite{smale1}, S.~Smale, following the ideas sketched out by E.~Routh in \cite{Routh2}, examined the stability of relative equilibria of simple mechanical
	systems with symmetries. For a general system of this kind, a \emph{relative equilibrium} is a dynamical fixed point (\ie an equilibrium point) in the \emph{reduced} phase space obtained by
	quotienting the original phase space by the symmetry group. Thus, generally speaking, relative equilibria are the analogue of fixed points for systems without symmetry (whence their great
	importance), yet they can also be viewed as one-parameter group orbits. Of course, the larger the symmetry group is, the richer the supply of relative equilibria becomes.
	For a system of particles in the plane described in the coordinates of the centre of mass, subject to the action of the rotation group $\SO(2)$ --- like the one that we examine here --- relative
	equilibria are solutions in which the whole system rotates with constant angular velocity around the barycentre. For this reason they are also called \emph{dynamical motions in steady
	rotation}.
	
	Given a relative equilibrium, it is natural to investigate its stability properties in order to understand the dynamical behaviour of the orbits nearby.
	Two of the main methods used to study the \emph{stability} of relative equilibria are the \emph{Energy-Casimir method} and the \emph{Energy-Momentum method}; however, even when
	applicable, they do not give any information about the \emph{instability} without further investigation. One of the few feasible methods to study the matter of stability is to show that the
	Hamiltonian $H$, or some other integral, has a maximum or minimum at a critical point: if the maximum or minimum is isolated then $H$ is a Lyapunov function and the equilibrium point is
	stable. Unfortunately, in the $n$-body context, it is easy to see (cf.~\cite[page~86]{Moeckel}) that this approach never works in the case of relative equilibria, and for this reason it is hopeless
	to try to prove their stability (or instability).
	Instead of that, we concentrate here on the notions of \emph{linear} and \emph{spectral stability} (see Subsection~\ref{subsec:linstabham} for their definition): we linearise the Hamiltonian
	system around a relative equilibrium and analyse its features. This involves the computation of the spectrum of a Hamiltonian matrix, which is symmetric with respect to both axes in the
	complex plane. A direct consequence of this fact is that relative equilibria are never asymptotically stable.
	
	In studying symmetric systems of particles it is usual to introduce the so-called \emph{augmented potential} $\mathcal{U}_\Xi$, which is equal to the 
	potential of the system plus a term coming from the centrifugal forces (cf.~\cite{MR1171218} and references therein). The reason is that relative equilibria are precisely the critical points
	of this modified potential (see \cite{smale1}).
		
	Our main result reads as follows (see Theorem~\ref{thm:mainequiv} and Section~\ref{sec:main} for a more precise statement, further details and the proof).
	
	\begin{mainthm}
		Let $\bar{x}$ be a critical point of the augmented potential and assume that it has even nullity. If $\iMor(\bar{x})$ is odd, then the relative equilibrium corresponding to $\bar{x}$ is
		spectrally unstable. 
	\end{mainthm}
	
	\noindent
	An immediate consequence is the following.
	
	\begin{maincor}
		Let $\bar{x}$ be a critical point of the augmented potential. If its Morse index or its nullity are odd then the corresponding relative equilibrium is linearly unstable.
	\end{maincor}

	The main tool that we use in the proof of this theorem is the \emph{spectral flow} (in the very elementary case of Hermitian matrices). We recall that this is a well-known integer-valued
	homotopy invariant of paths of self-adjoint Fredholm operators introduced by M.~F.~Atiyah, V.~K.~Patodi and I.~M.~Singer in \cite{MR0397799}.
	In finite-dimensional situations it is nothing else but the difference of the Morse index at the endpoints (see Section~\ref{sec:auxiliary} for its definition and Appendix~\ref{app:A} for its main
	properties). Up to perturbation, non-degeneracy and transversality conditions, this invariant can be computed in terms of the so-called \emph{crossing forms}, which, intuitively speaking, counts
	in an appropriate way the net number of eigenvalues crossing the value $0$ in a transversal way.
	In our setting this need not be the case; however, the third author developed in other papers (see for instance \cite{MR2057171}, or Appendix~\ref{app:A} for a short
	description) a non-perturbative analysis of the non-transversal intersections. The reason behind the choice of a non-perturbative technique lies in the fact that, in general, perturbative methods
	preserve global invariants but completely destroy the local information concerning the single intersection. By means of this theory, based on what has been termed \emph{partial signatures},
	we have been able to prove Theorem~\ref{thm:mainequiv}.
	
	The main applications of our result (see Section~\ref{sec:nbody}) are directed towards the $\alpha$-homogeneous and the logarithmic potentials, although also some other interesting classes
	can be reformulated in our framework, such as the Lennard-Jones interaction potential. Our theorem offers indeed a unifying viewpoint of all these quite different situations, since the property
	that it unravels descends only from the rotational invariance of the mechanical system. All of these potentials are extensively studied in literature: the $\alpha$-homogeneous ones are the
	natural generalisation of the gravitational attraction ($\alpha = 1$) and they are employed in different atomic models, whilst the logarithmic potential naturally arises when looking from a
	dynamical viewpoint at the stationary helicoidal solutions of the $n$-vortex filaments model, which is popular and useful in Fluid Mechanics. See \cite{MR1329405, Venturelli, MR2245515} and
	references therein for the homogeneous cases, \cite{MR3007736} for the logarithmic one and \cite{MR2439573} for a general overview.
	
	To detect a relative equilibrium in an $n$-body-type problem means to determine a moving \emph{planar central configuration} of the bodies which solves
	Newton's equations and in which the attractive force is perfectly balanced by the centrifugal one. This is currently the only way known to obtain exact solutions, albeit finding central
	configurations amounts to solving a system of highly nonlinear algebraic equations and is therefore very hard (see \cite{Moeckel} for the Newtonian case and \cite{MR2422381} for the
	$\alpha$-homogeneous one). 

	Being invariant under the symmetry group of Euclidean transformations and admitting linear momentum, angular momentum and energy as first integrals, $n$-body-type problems are highly
	degenerate. This in particular yields Jacobians with nullity $8$ (cf.~\cite{MR1736548, MR2468466} for the gravitational force), but only in an inertial reference frame: indeed, if we move (as we
	do) to a suitable uniformly rotating coordinate system (so that the relative equilibrium becomes an effective equilibrium) six out of the eight eigenvalues produced by the first integrals depend
	on the angular velocity. This is not surprising at all, since linear stability properties strongly depend on the choice of the frame of the observer.
	For this reason, studying the case $\alpha = 1$, R.~Moeckel in \cite{Moeckel} defined the linear and spectral stability by ruling out all the eigenvalues linked to this kind of degeneracy.
	In the same context, K.~R.~Meyer and D.~S.~Schmidt concluded in \cite{MeySch} a deep study of the linearised equations: in particular, they introduced a suitable system of symplectic
	coordinates in which the matrices are block-diagonal, with one block representing the translational invariance of the problem and another one carrying the symmetries induced by dilations and
	rotations. These two submatrices generate the eight eigenvalues responsible of degeneracy, whilst a third (and last) block contains all the information about stability, in the sense mentioned
	above. We observe that an analogous decomposition holds also for the potentials that we examine (see Subsection~\ref{subsec:decomp}).
		
	In this picture, it is worthwhile to mention a conjecture on linear stability stated by Moeckel, which we report here.
		\begin{quote}
			\textbf{Moeckel's Conjecture} (cf.~\cite[Problem~16]{MR2970201}) --- In the planar Newtonian $n$-body problem, the central configuration associated with a linearly stable
			relative equilibrium is a non-degenerate minimum of the potential function restricted to the \emph{shape sphere} (\ie the $\SO(2)$-quotient of the ellipsoid of inertia).
		\end{quote}
	This conjecture is still unproved; however, X.~Hu and S.~Sun have made some progress. More precisely, they showed in \cite{HuSunCRASP} that if the Morse index or the nullity of a central
	configuration (viewed as a critical point of the potential restricted to the shape sphere) are odd, then the corresponding relative equilibrium is linearly unstable. Therefore the central
	configurations giving rise to linearly stable relative equilibria should correspond to a critical point with even Morse index and nullity. The main result in \cite{HuSunCRASP} is the first attempt
	towards the understanding of the relationship (if there is any) between two dynamics: the gradient flow on the shape sphere and Hamilton's equations in the phase space.
	
	The contribution of our paper in this setting is twofold:
		\begin{enumerate}
			\item We provide a complete and detailed proof of the result on linear instability proved in \cite{HuSunCRASP} and we extend it to a very general class of interaction potentials by
				using spectral flow techniques.
		    
			\item We prove a result on spectral instability by means of the theory of partial signatures previously developed in \cite{MR2057171}. Note that our Corollary~\ref{cor:main} is actually
				the main result in \cite{HuSunCRASP} (written there in the gravitational setting only).
		\end{enumerate}
	Moeckel's Conjecture can thus be adapted to the class of potentials that we study; accordingly, we reformulate it as follows:
		\begin{quote}
			\textbf{Conjecture} --- In planar $\SO(2)$-symmetric mechanical systems, a critical point of the augmented potential associated with a linearly stable relative equilibrium, is a
			non-degenerate minimum.
		\end{quote}
	We cast some light on this question with Theorem~\ref{thm:mainequiv} and with Theorem~\ref{thm:main} in the special case of $n$-body-type problems.

	Furthermore, following the approach of G.~E.~Roberts in \cite{MR1709850}, we are able to give a sufficient condition for spectral instability of a relative equilibrium
	(at least in the $\alpha$-homogeneous case) in terms of the potential evaluated at a central configuration. It is in fact rather foreseeable that the linear stability depends also on the
	homogeneity parameter $\alpha$ (see Subsection~\ref{subsec:linstab} and cf.~Corollary~\ref{cor:ineqUhom} for a precise statement).

	\begin{mainthm}
		Let $\bar{x}$ be a central configuration. If the following inequality holds
			\[
				\sum_{\substack{i, j = 1 \\ i < j}}^{n} \frac{m_i + m_j}{\abs{\bar{x}_i - \bar{x}_j}^{\alpha + 2}} > \frac{2n + \alpha - 4}{\alpha}\ \Uhom(\bar{x})
			\]
		then the arising relative equilibrium is linearly unstable.
	\end{mainthm}
	
	We conclude this section by pointing out that no sufficient condition for detecting the linear or spectral stability has been found thus far. This question is addressed in a forthcoming paper
	\cite{BJP2}, where we are trying to establish in a precise way the stability properties of the relative equilibria by using some symplectic and variational techniques, mainly based on the Maslov
	index, index theorems and topological invariants.
	
	The following table of contents shows how the paper is organised.
	
	\tableofcontents

	\section{Description of the problem: setting and preliminaries} \label{sec:description}
	
	In this section we briefly outline the basic definitions and properties of simple mechanical systems with symmetry, as well as their reduction to the quotient space. 
	
	Consider the Euclidean plane $\R^2$ endowed with the usual inner product $\langle \cdot, \cdot \rangle$ and let $m_1,\dots, m_n$ be $n \geq 3$ positive real numbers which can be 
	thought of as masses.
	The configuration space of $n$ point particles with masses $m_i$, with $i \in \{1, \dotsc, n\}$, will therefore be a suitable subset $X \subseteq \R^{2n}$ (equipped with its Euclidean inner
	product, which we denote again by $\langle \cdot, \cdot \rangle$). For any position vector $q \= \trasp{(q_1, \dots, q_n)} \in \R^{2n}$, with $q_i \in \R^2$ (column vector) for
	every $i \in \{1, \dotsc, n\}$, we can define a norm in $\R^{2n}$ through the \emph{moment of inertia}:%
		\footnote{What we define here is actually the double of the moment of inertia: we drop the factor $1/2$ in order to make computations lighter in the following.}
		\[
			\I(q) \= \norm{q}_M^2 \= \langle Mq, q \rangle = \sum_{i = 1}^n m_i \abs{q_i}^2,
		\]
	where $M \in \Mat(2n, \R)$ is the diagonal \emph{mass matrix} $\diag(m_1 I_2, \dotsc, m_n I_2)$, $I_n$ is the $n \times n$ identity matrix and $\abs{q_i}$ denotes the Euclidean norm of
	$q_i$ in $\R^2$.
	
	A simple mechanical system of $n$ point particles on $X$ is described by a \emph{Lagrangian function} $\mathscr{L}: TX \to \R$ of the form
		\[
			\mathscr{L}(q, \dot{q}) \= \mathcal{K}(q, \dot{q}) + \mathcal{U}(q),
		\]
	where $\mathcal{K} : TX \to \R$ is the \emph{kinetic energy} of the system and $\mathcal{U} : X \to \R$ is its \emph{potential function}. This Lagrangian thus equals the difference between the
	kinetic energy and the potential energy ($-\mathcal{U}$); in our case we have $\mathcal{K}(q, \dot{q}) \= \frac{1}{2} \I(\dot{q})$.

	Using the mass matrix $M$, Newton's equations can be written as the following second-order system of
	ordinary differential equations on $X$:
		\begin{equation}\label{eq:Newton}
			M \ddot{q} = \nabla \mathcal{U}(q),
		\end{equation}
	which can of course be transformed into a first-order system as follows. Let us introduce the \emph{Hamiltonian function} $\mathscr{H} : T^*X \to \R$, defined by
		\[
			\mathscr{H}(q, p) \= \frac{1}{2}\langle M^{-1}\trasp{p}, \trasp{p} \rangle - \mathcal{U}(q).
		\]
	Here $p \= (p_1, \dotsc, p_n) \in \R^{2n}$, with $p_i \in \R^2$ (row vector) for all $i \in \{ 1, \dotsc, n \}$, is the linear momentum conjugate to $q$.
	The Hamiltonian system associated with \eqref{eq:Newton} is the first-order system of ordinary differential equations on the phase space $T^*X \cong X \times \R^{2n}$ given by
		\begin{equation}\label{eq:Hamilton}
			\begin{cases}
				\dot{q} = \partial_p \mathscr{H} =M^{-1} \trasp{p}  \\
				\trasp{\dot{p}} = -\partial_q \mathscr{H} = \nabla \mathcal{U}(q).
			\end{cases}
		\end{equation}
	We shall consider simple mechanical systems with an \emph{$\SO(2)$-symmetry}, meaning that the group $\SO(2)$ acts properly on $X$ through isometries that leave the potential function
	$\mathcal{U}$ unchanged. It follows that the Lagrangian and the Hamiltonian are $\SO(2)$-invariant under the natural lift of this action to $TX$ and to $T^*X$, respectively.

		\subsection{Relative equilibria} \label{subsec:releq}

	Among all the  solutions of Newton's Equations~\eqref{eq:Newton}, as already observed,  the simplest are represented by a special class of periodic solutions called \emph{relative
	equilibria}.
	
	In the following and throughout all this paper, the matrix
		\[
			J_{2n} \= \begin{pmatrix}
					0 & -I_n \\
					I_n & 0
				\end{pmatrix}
		\]
	will denote the complex structure in $\R^{2n}$, but it will always be written simply as $J$, its dimension being clear from the context.
	
	Let $e^{\omega J t} = \begin{pmatrix} \cos\omega t & -\sin\omega t \\ \sin\omega t & \cos\omega t \end{pmatrix}$ be the matrix representing the rotation in the plane with angular velocity
	$\omega$. In order to rewrite Hamilton's Equations~\eqref{eq:Hamilton} in a frame uniformly rotating about the origin with period $2\pi/\omega$, we employ the following symplectic change of
	coordinates:
		\begin{equation*}
			\begin{cases}
				x \= R(t)\, q\\
				\trasp{y} \= R(t)\, \trasp{p}
			\end{cases}
		\end{equation*}
	where $R(t)$ is the $2n \times 2n$ block-diagonal matrix $\diag_n(e^{\omega J t}, \dots, e^{\omega J t})$. Since a symplectic change of variables preserves the Hamiltonian structure, in
	these new coordinates System~\eqref{eq:Hamilton} is still Hamiltonian and transforms as follows:
		\begin{equation}\label{eq:newHamilton}
			\begin{cases}
				\dot{x} = \partial_y \hat{\mathscr{H}} = \omega Kx + M^{-1} \trasp{y}  \\
				\trasp{\dot{y}} = -\partial_x \hat{\mathscr{H}} = \nabla \mathcal{U}(x) + \omega K \trasp{y}
			\end{cases}
		\end{equation}
	where $K$ is the $2n \times 2n$ block-diagonal matrix $\diag_n(J, \dots, J)$ and $\hat{\mathscr{H}}$ is the new Hamiltonian function given by
		\begin{equation} \label{eq:augmHam}
			 \hat{\mathscr{H}}(x, y) \= \frac{1}{2} \langle M^{-1} \trasp{y}, \trasp{y} \rangle - \mathcal{U}(x) + \omega \langle K x, \trasp{y} \rangle.
		\end{equation}
	From the physical point of view, the term involving $K$ is due to the Coriolis force.

	An equilibrium for System~\eqref{eq:newHamilton} must satisfy the conditions
		\[
			\begin{cases}
				\omega K x + M^{-1} \trasp{y} = 0 \\
				\nabla \mathcal{U}(x) + \omega K \trasp{y} = 0,
			\end{cases}
		\]
	which, taking into account that $[K,M] = 0$ and that $K^2 = -I$, can be rewritten as
		\begin{equation}\label{eq:equilibrium2}
			\begin{cases}
				\trasp{y} = -\omega M K x \\
				M^{-1} \nabla \mathcal{U}(x) + \omega^2 x = 0.
			\end{cases}
		\end{equation}
	
	Setting now $\Xi \= \omega K$, it is easy to see that the Hamiltonian
	$\hat{\mathscr{H}}$ defined in \eqref{eq:augmHam} coincides with the \emph{augmented Hamiltonian function}
		\[
			\mathscr{H}_\Xi (x, y) \= \mathcal{K}_\Xi(x, y) - \mathcal{U}_\Xi(x),
		\]
	where
		\[
			\mathcal{K}_\Xi(x, y) \= \frac{1}{2} \norm{M^{-1}\trasp{y} + \Xi x}^2_M
		\]
	is the \emph{augmented kinetic energy} and
		\begin{equation} \label{eq:augpotential}
			\mathcal{U}_\Xi(x) \= \mathcal{U}(x) + \frac{1}{2} \norm{\Xi x}_M^2.
		\end{equation}
	is called the \emph{augmented potential function}. In terms of these augmented quantities, System~\eqref{eq:equilibrium2} becomes
		\[
			\begin{cases}
				\trasp{y} = - M\Xi x \\
				\nabla \mathcal{U}_\Xi(x) = 0
			\end{cases}
		\]
	and we have the following definition.
		
	\begin{defn}
		The point $(\bar{x}, \bar{y}) \in T^*X$ is a \emph{relative equilibrium} for Newton's Equations~\eqref{eq:Newton} with potential $\mathcal{U}$ if both the following conditions hold:
			\begin{enumerate}[1)]
				\item $\trasp{\bar{y}} = -M \Xi \bar{x}$;
				\item $\bar{x}$ is a critical point of the augmented potential function $\mathcal{U}_\Xi$.
			\end{enumerate}
	\end{defn}
	
	Let us now consider the autonomous Hamiltonian System~\eqref{eq:newHamilton} in $\R^{4n}$: by grouping variables into $z \= \trasp{(\trasp{x}, y)}$, it can be written as follows:
		\begin{equation} \label{eq:hamnuovevariabili}
			\dot{z}(t) = -J \nabla \hat{\mathscr{H}} \big( z(t) \big).
		\end{equation}
	Linearising it at the relative equilibrium $\bar{z} \= \trasp{(\trasp{\bar{x}}, \bar{y})}$, we obtain the linear autonomous Hamiltonian system
		\begin{equation} \label{eq:hamsyslinearized}
			\dot{z}(t) = -JB z(t),
		\end{equation}
	where $B$ is the constant $4n \times 4n$ symmetric matrix given by
		\begin{equation} \label{eq:generalB}
			B \= \begin{pmatrix}
				-D^2 \mathcal{U}(\bar{x}) & \trasp{\Xi} \\
				\Xi & M^{-1}
			\end{pmatrix}.
		\end{equation}

		\subsection{Linear and spectral stability for autonomous Hamiltonian systems} \label{subsec:linstabham}
	
	We now recall some basic definitions and well-known facts about the linear stability of autonomous Hamiltonian systems, starting with the definition of the symplectic group and its Lie algebra.
	The reader is invited to consult, for instance, \cite{Abbo} for more details.

	The \emph{(real) symplectic group} is the set
		\[
			\Sp(2n, \R) \= \Set{ S \in \Mat(2n, \R) | \trasp{S}JS = J}.
		\]
	Symplectic matrices correspond to symplectic automorphism of the standard symplectic space $(\R^{2n}, \Omega)$, where $\Omega$ is the standard symplectic form represented by $J$
	via the standard inner product of $\R^{2n}$, \ie $\Omega(u, v) \= \langle Ju, v\rangle$ for every $u, v \in \R^{2n}$.
	
	By differentiating the equation $\trasp{H}JH = J$ and evaluating it at the identity matrix, we find the characterising relation of the Hamiltonian matrices: the Lie algebra of the symplectic group
	is defined as
		\[
			\ssp(2n, \R) \= \Set{ H \in \Mat(2n, \R) | \trasp{H}J + JH = 0 },
		\]
	and its elements are called \emph{Hamiltonian} or \emph{infinitesimally symplectic}.
	
	\begin{rmk} \label{rmk:exp}
		Since $\Sp(2n, \R)$ is a matrix Lie group and $\ssp(2n, \R)$ is its Lie algebra, the exponential map $\exp : \ssp(2n, \R) \to \Sp(2n, \R)$ coincides with the usual matrix exponential, and
		therefore we have that $H$ is a Hamiltonian matrix if and only if $\exp(H)$ is symplectic. It follows that $\lambda \in \sigma(H)$ if and only if $e^\lambda \in \sigma\bigl( \exp(H) \bigr)$.
	\end{rmk}
	
	The next proposition recollects the symmetries of the spectra of Hamiltonian and symplectic matrices.
	
	\begin{prop} \label{prop:spectra}
		The characteristic polynomial of a symplectic matrix is a reciprocal polynomial. Thus if $\lambda$ is an eigenvalue of a real symplectic matrix, then so are $\lambda^{-1}$,
		$\overline{\lambda}$, $\overline{\lambda}^{-1}$.
		
		The characteristic polynomial of a Hamiltonian matrix is an even polynomial. Thus if $\lambda$ is an eigenvalue of a Hamiltonian matrix, then so are $-\lambda$, $\overline{\lambda}$,
		$-\overline{\lambda}$.
	\end{prop}
	
	\begin{proof}
		See \cite[Proposition~3.3.1]{MR2468466}.
	\end{proof}
	
	\begin{rmk} \label{rmk:spectra}
		It descends directly from Proposition~\ref{prop:spectra} that the spectrum of a Hamiltonian matrix $H$ is, in particular, symmetric with respect to the real axis of the complex
		plane. Moreover, $0$ has always even (possibly zero) algebraic multiplicity as a root of the characteristic polynomial of $H$.
	\end{rmk}
	
	We now present the definition of spectral and linear stability for Hamiltonian matrices, in view of the fact that these are the ones on which we shall focus in our analyses.
	
	\begin{defn}\label{def:stabsym}
		A Hamiltonian matrix $H \in \ssp(2n, \R)$ is said to be \emph{spectrally stable} if $\sigma(H) \subset i\R$, whereas it is \emph{linearly stable} if $\sigma(H) \subset i\R$ and in addition
		it is diagonalisable.
	\end{defn}
	
	This concept is easily adapted to symplectic matrices by using the exponential map, as explained in Remark~\ref{rmk:exp}, and by remembering that the imaginary axis of the complex plane
	is the Lie algebra of the unit circle $\U$ in the same plane (cf.~Remark~\ref{rmk:spectra}). Indeed, a symplectic matrix $S$ is said to be \emph{spectrally stable} if $\sigma(S) \subset \U$ and,
	as before, the property of \emph{linear stability} requires in addition the diagonalisability of $S$.
	
	A linear autonomous Hamiltonian system in $\R^{2n}$ has the form
		\begin{equation}\label{eq:Hamsys}
			 \dot{\zeta}(t) = JA \zeta(t),
		\end{equation}
	where $A$ is a symmetric matrix. Being it autonomous, its fundamental solution can be written in the explicit form
		\[
			\gamma(t) \= \exp(tJA).
		\]
	
	The definition of spectral and linear stability for this kind of systems is given in accord with Definition~\ref{def:stabsym}.
	
	\begin{defn}
		The linear autonomous Hamiltonian System~\eqref{eq:Hamsys} is \emph{spectrally} (resp.~\emph{linearly}) \emph{stable} if the symplectic matrix $\exp(JA)$ corresponding to its
		fundamental solution at time $t = 1$ is spectrally (resp.~linearly) stable. We say that System~\eqref{eq:Hamsys} is \emph{degenerate} if $0 \in \sigma(JA)$, or equivalently if
		$1 \in \sigma\bigl( \exp(JA) \bigr)$, and \emph{non-degenerate} otherwise.
	\end{defn}
	
	We conclude the subsection by reporting a criterion for linear stability of symplectic matrices, in order to complete our brief recollection of definitions and results on this topic. We also
	point out that we are not aware of any existing proof of this lemma. In the following, the symbol $\opnorm{\,\cdot\,}{\H}$ will denote the norm of a bounded linear operator from the Hilbert
	space $\H$ to itself.
	
	\begin{lem}
		A matrix $S \in \Sp(2n, \R)$ is linearly stable if and only if
			\[
				\sup_{m \in \N} \opnorm{S^m}{\R^{2n}} < +\infty.
			\]
	\end{lem}
 	
 	\begin{proof}
 		If $S$ is linearly stable, then in particular it is similar to a diagonal matrix $D$ through an invertible matrix $P$, so that we have
			\[
				\begin{split}
					\sup_{m \in \N} \opnorm{S^m}{\R^{2n}} & = \sup_{m \in \N} \opnorm{P^{-1}D^mP}{\R^{2n}} \\
													& \leq  \opnorm{P^{-1}}{\R^{2n}} \opnorm{P}{\R^{2n}} \sup_{m \in \N} \opnorm{D^m}{\R^{2n}} \\
													& = \opnorm{P^{-1}}{\R^{2n}} \opnorm{P}{\R^{2n}} < +\infty,
				\end{split}
			\]
		where the last equality holds true because all the eigenvalues of $S$ (and hence those of $D$) lie on the unit circle.
		
		Vice versa, if $S$ is not linearly stable then it is not spectrally stable or it is not diagonalisable (or both). If it is spectrally unstable there exists, by definition, at least one eigenvalue
		$\lambda \notin \U$, and we can assume, by the properties of the spectrum of symplectic matrices, that $\abs{\lambda} > 1$. Writing $S$ in its Jordan form (possibly diagonal) and
		computing $S^m$ yields on the diagonal a power $\lambda^m$, whose modulus diverges as $m \to +\infty$. Hence $\opnorm{S^m}{\R^{2n}} \to +\infty$.
		If $S$ is not diagonalisable, then there exists at least one Jordan block of size $k \geq 2$ (say) relative to the eigenvalue $\lambda$. Its $m$-th power has the form
			\[
				\begin{bmatrix}
					\lambda^m & m\lambda^{m - 1} & 0 & \hdotsfor{2} & 0 \\
					0 & \lambda^m & m\lambda^{m - 1} & 0 & \ldots & 0 \\
					\hdotsfor{6} \\
					\hdotsfor{6} \\
					0 & \hdotsfor{2} & 0 & \lambda^m & m\lambda^{m - 1} \\
					0 & \hdotsfor{3} & 0 & \lambda^m
				\end{bmatrix},
			\]
		and therefore even in this case (regardless of the fact that $\lambda \in \U$ or not) the norm of $S^m$ tends to $+\infty$ as $m$ goes to $+\infty$.
 	\end{proof}

	\section{Auxiliary results}  \label{sec:auxiliary}
	
	In this section we present the lemmata and the propositions needed in the proof of the main results in Section~\ref{sec:main}.
	We first introduce some notation and definitions; for further properties we refer to Appendix~\ref{app:A}.

		\subsection{Notation and definitions} \label{subsec:spfl}
	
	Let $\H$ be, throughout all this paper, a finite-dimensional complex Hilbert space (we shall specify its dimension when needed).
	We denote by $\B(\H)$ the Banach algebra of all (bounded) linear operators $T : \H \to \H$ and by $\Bsa(\H)$ the subset of all (bounded) linear self-adjoint operators on $\H$. For a
	subset $\A \subseteq \B(\H)$, the writing $\G\A$ indicates the set of all invertible elements of $\A$.
	
	\begin{defn}
		For any $T \in \Bsa(\H)$, we define its \emph{index} $n^-(T)$, its \emph{nullity} $\nu(T)$ and its \emph{coindex} $n^+(T)$ as the numbers of its negative, null and positive eigenvalues,
		respectively. Its \emph{extended index} and the \emph{extended coindex} are defined as
			\[
				\nmext(T) \= n^-(T) + \nu(T), \qquad \npext(T) \= n^+(T) + \nu(T).
			\]
		The \emph{signature} $\sgn(T)$ of $T$ is the difference between its coindex and its index:
		 	\[
				\sgn(T) \= n^+(T) - n^-(T).
			\]
	\end{defn}
	
	\begin{rmk}
		We shall refer to the index $n^-(T)$ of a self-adjoint operator $T \in \Bsa(\H)$ also as its \emph{Morse index}, which will be denoted by $\iMor(T)$.
	\end{rmk}
	
	\begin{defn}
		Let $X$ be a topological space, $Y \subseteq X$ a subspace and $a, b \in \R$, with $a < b$. We denote by $\Omega(X, Y)$ the set of all continuous paths $\gamma : [a, b] \to X$ with
		endpoints in $Y$. Instead of $\Omega(X, X)$ we simply write $\Omega(X)$. Two paths $\gamma, \delta \in \Omega(X,Y)$ are said to be \emph{(free) homotopic} if there is a continuous
		map $F : [0,1] \times [a, b] \to X$ which satisfies the following properties:
			\begin{enumerate}[i)]
				\item $F(0,\cdot) = \gamma$, $F(1,\cdot) = \delta$;
				\item $F(s, a) \in Y$, $F(s, b) \in Y$ for all $s \in [0,1]$.
			\end{enumerate}
		The set of homotopy classes in this sense is denoted by $\tilde{\pi}_1(X,Y)$.
	\end{defn}
	
	\begin{rmk}
		Note that the endpoints are not fixed along the homotopy; however, they are allowed to move only within $Y$.
	\end{rmk}
	
	Taking into account \cite[Corollary~3.7]{Les05}, we are entitled to give the following definition:
	  
	\begin{defn}\label{def:spectralflow}
		Let $a, b \in \R$, with $a < b$, and let $T \in \Omega \bigl( \Bsa(\H), \G\Bsa(\H) \bigr)$. We define its \emph{spectral flow} on the interval $[a, b]$ as:
			\[
				\spfl \big( T, [a, b] \big) \= \npext \big( T(b) \big) - \npext \big( T(a) \big).
	 		\]
	\end{defn}
	
	\begin{rmk} \label{rmk:spfl}
		It is worthwhile noting that
			\[
				\spfl \big( T, [a, b] \big) = n^- \big( T(a) \big) - n^- \big( T(b) \big).
			\]
	\end{rmk}
	
	We now switch to introduce the key notion of \emph{crossing}.
	
	\begin{defn} \label{def:crossing}
		Let $a, b \in \R$, with $a < b$, and let $T \in \Cscr^1\bigl( [a, b], \Bsa(\H) \bigr)$. A \emph{crossing instant} (or simply a \emph{crossing}) for the path $T$ is a number $t_* \in [a, b]$ for
		which $T(t_*)$ is not injective. We define the \emph{crossing operator} (also called \emph{crossing form}) $\Gamma(T, t_*) : \ker T(t_*) \to \ker T(t_*)$ of $T$ with respect to the crossing
		$t_*$ by
			\begin{equation} \label{eq:defcrossform}
				\Gamma(T, t_*) \= Q\dot{T}(t_*)Q \bigr|_{\ker T(t_*)},
			\end{equation}
		where $Q : \H \to \H$ denotes the orthogonal projection onto the kernel of $T(t_*)$. A crossing $t_*$ is called \emph{regular} if the crossing form $\Gamma(T, t_*)$ is non-degenerate.
		We say that the path $T$ is \emph{regular} if each crossing for $T$ is regular.
	\end{defn}
	
	\begin{rmk} \label{rmk:crossform}
		The computation of the spectral flow of a path of operators involves the signature of the crossing form. We point out here that we actually refer to the signature of the \emph{quadratic
		form} associated with the linear map defined in \eqref{eq:defcrossform}, that is, we make the following implicit identification. Given an endomorphism $\Gamma : V \to V$ on a vector
		space $V$, it is associated in a natural way with a bilinear form $\mathcal{B}_\Gamma : V \times V^* \to \R$ defined by
			\[
				\mathcal{B}_\Gamma(u, f) \= f(\Gamma u),
			\]
		where $f \in V^*$ is an element of the dual space $V^*$ of $V$. Since $V^* \cong V$ one can then define
			\[
				\mathcal{B}_\Gamma(u, v) \= \trasp{v} \Gamma u.
			\]
		The quadratic form associated with $\Gamma$ is thus the quadratic form associated with $\mathcal{B}_\Gamma$. This is the justification for the abuse of language and notation that the
		reader will encounter throughout the paper.
	\end{rmk}
	
	As last piece of information, we point out that in the rest of the paper we shall denote the matrix $iJ$ by $G$.

		\subsection{Relationships among linear stability, spectral flow and partial signatures}
	
	Here are the properties and facts that we shall exploit later to prove our main theorem. In this subsection we identify the Hilbert space $\H$ with $\C^{4n}$ and consider the affine path
	$D : [0, +\infty) \to \Bsa(\C^{4n})$ defined by
		\[
			D(t) \= A + tG,
		\]
	where $A \in \Bsa(\C^{4n})$ is a real symmetric matrix (hence $JA$ is Hamiltonian). Without different indication, it will be understood that $\H$, $A$ and $D$ are as defined above.
	
	Thanks to the identification $\H = \C^{4n}$, we implicitly fix the canonical basis of $\C^{4n}$ and therefore every operator in $\Bsa(\H)$ is represented by a $4n \times 4n$ complex
	Hermitian matrix.
	
	We explicitly note that the spectral flow does not depend on the particular inner product chosen but only on the associated quadratic form (see \cite{arxiv}).
	
	\begin{lem}
		Assume that $JA$ is linearly stable.
		
		Then if $A$ is singular there exist $\eps > 0$ and $T > \eps$ such that 
			\begin{enumerate}[(i)]
				\item The instant $t_* = 0$ is the only crossing for the path $D$ on $[0, \eps]$;
				\item $\spfl \bigl( D, [\eps, T_1] \bigr) = \spfl \bigl(D, [\eps, T_2] \bigr)$ for all $T_1, T_2 \geq T$;
				\item $\spfl \bigl( D, [\eps, T] \bigr)$ is an even number.
			\end{enumerate}
		If $A$ is non-singular there exists $T > 0$ such that
			\begin{enumerate}[(i)]
				\item $\spfl \bigl( D, [0, T_1] \bigr) = \spfl \bigl(D, [0, T_2] \bigr)$ for all $T_1, T_2 \geq T$;
				\item $\spfl \bigl( D, [0, T] \bigr)$ is an even number.
			\end{enumerate}
	\end{lem}
	
	\begin{proof}
		Since $A$ is symmetric, the matrix $JA$ is Hamiltonian. Therefore its spectrum is symmetric with respect to the real axis of the complex plane and $\ker A$ (which is equal to $\ker JA$
		because $J$ is an isomorphism) is even-dimensional, being $JA$ diagonalisable. Furthermore, due to the Krein properties of $G$ (see Subsection~\ref{subsec:Krein}), the crossing form
		$Q_\lambda G Q_\lambda|_{E_\lambda}$ is always non-degenerate on each eigenspace $E_\lambda$.
		
		Hence the hypotheses of Proposition~\ref{prop:MainA} or of Corollary~\ref{cor:MainA} (depending whether $A$ is invertible or not) are fulfilled and this proof reduces to the
		corresponding one in Appendix~\ref{app:A}.
	\end{proof}
	
	\begin{prop} \label{thm:sgnpartial}
		Assume that $t_* > 0$ is an isolated (possibly non-regular) crossing instant for the path $D$. Then, for $\eps > 0$ small enough, 
			\[
				\spfl \bigl( D, [t_* - \eps, t_* + \eps] \bigr) = \sgn \mathcal{B}_1,
			\]
		 where
			\[
				\mathcal{B}_1 \= \langle G\, \cdot, \cdot \rangle \bigr|_{\H_{t_*}}
			\]
		and $\H_{t_*}$ is the generalised eigenspace given by
			\[
				\H_{t_*} \= \bigcup_{j=1}^{4n} \ker (GA+ t_* I)^j.
			\]
	\end{prop}

	\begin{proof}
		We observe that for $t \in(0, +\infty)$ 
			\[
				D(t) = t \biggl( \frac{1}{t}A + G \biggr) \eq t\, \tilde{D} \biggl( \frac{1}{t} \biggr).
			\]
		Clearly, the spectral flow is invariant by multiplication of  a path for a positive real-analytic function:
			\[
				\spfl \bigl( D, [t_* - \eps, t_* + \eps] \bigr) = \spfl \biggl( \tilde{D}, \biggl[ \frac{1}{t_*} - \eps, \frac{1}{t_*} + \eps \biggr] \biggr).
			\]
		Using now Proposition~\ref{thm:cruciale}, with $C \= G$ and $s \= \frac{1}{t}$, we obtain the thesis (observe that the difference in sign to the local contribution to the spectral flow is due to the change of variable $s \= \frac{1}{t}$).
	\end{proof}
	
	We now prove the main result of this section by means of the theory of partial signatures (see Subsection~\ref{subsec:parsgn}).
	
	\begin{thm} \label{prop:mainB}
		If $JA$ is spectrally stable, then $n^-(A)$ is even.
	\end{thm}
	
	\begin{proof}
		If we write
			\[
				D(t) = -J(JA - itI)
			\]
		we see that $t_* \in [0, +\infty)$ is a crossing instant for $D$ if and only if
			\[
				it_* \in \sigma(JA) \cap i[0, +\infty).
			\]
		Indeed, since $-J$ is an isomorphism,
			\[
				\ker D(t) = \ker (JA - itI) \qquad \forall t \in [0, +\infty),
			\]
		and thus there is a bijection between the set of crossing instants $t_*$ of $D$ and the set of pure imaginary eigenvalues of $JA$ of the form $it_*$.
		Being $D$ an affine path, it is real-analytic, and the Principle of Analytic Continuation implies that every crossing (be it regular or not) is isolated, because it can be regarded as a zero
		of the (real-analytic) map $\det D(t)$.
		
		Let us examine the strictly positive crossings. By Proposition~\ref{thm:sgnpartial}, in a suitable neighbourhood with radius $\delta > 0$ around a crossing $t_* > 0$
		we see that
			\[
				\spfl \bigl( D, [t_* - \delta, t_* + \delta] \bigr) = \sgn \mathcal{B}_1,
			\]
		where $\mathcal{B}_1$ and $\H_{t_*}$ are as in the aforementioned proposition. Furthermore, by the general theory of the Krein signature (see Subsection~\ref{subsec:Krein}), for any
		crossing $t_* \in (0,+\infty)$ the restriction $\langle G\, \cdot, \cdot \rangle|_{\H_{t_*}}$ of the Krein form to each generalised eigenspace $\H_{t_*}$ is non-degenerate.
		In particular, Remark~\ref{sgndim} yields
			\begin{equation} \label{eq:sfisolcross}
				\spfl \bigl( D, [t_* - \delta, t_* + \delta] \bigr) \equiv \dim \H_{t_*} \quad \mod 2.
			\end{equation}
		for every strictly positive crossing instant $t_*$.
		
		When turning our attention to the instant $t = 0$, we have to distinguish two situations: one where $A$ is singular and one where it is not. Let us start with the former and assume that
		$A$ is non-invertible, so that $t_* = 0$ is a crossing for the path $D$. Since this is isolated, by arguing as in the proof of (T2) in Proposition~\ref{prop:MainA} we can find $\eps > 0$ and
		$T > \eps$ such that the path $D$ has only $t_* = 0$ as crossing instant on $[0, \eps]$ and $\spfl \bigl( D, [\eps, T_1] \bigr) = \spfl \bigl( D, [\eps,T_2] \bigr)$ for every $T_1, T_2 \geq T$.
		Thus, recalling Remark~\ref{rmk:spfl} and the fact that $n^-\bigl( D(T) \bigr) = n^-(G)$, we obtain
			\begin{equation}\label{eq:finalespfl2}
				\begin{split}
					\spfl \bigl( D, [\eps, T] \bigr) &= n^-\bigl( D(\eps) \bigr) - n^- \bigl( D(T) \bigr) \\ 
										&= n^- \bigl( D(\eps) \bigr) - 2n.
				\end{split}
			\end{equation}
		We observe that the dimension of the generalised eigenspace $\H_0$ (which coincides with the algebraic multiplicity of the eigenvalue $0$) is even, being $JA$ Hamiltonian.
		Intuitively speaking, then, since the Krein form is non-degenerate on this subspace, the null eigenvalues move from $0$ as $t$ leaves $0$; and since its signature at the initial instant is
		$0$ (by Krein theory, see Appendix~\ref{app:A}, page~\pageref{page:sign0}), they split evenly: half become positive and half negative. This justifies the choice of $\eps$ so small that
			\begin{equation}\label{eq:92}
				n^-\bigl( D(\eps) \bigr) = n^-(A) + \frac{\dim \H_0}{2}.
			\end{equation}
		On the other hand, we have
			\[
				4n = \quad 2\ \sum_{\mathclap{it_*\, \in\, \sigma(JA)\, \cap\, i(0, +\infty)}}\ \dim \H_{t_*} + \dim \H_0,
			\]
		or, equally well,
			\begin{equation}\label{eq:9piu2}
				2n - \frac{\dim \H_0}{2} = \quad\ \sum_{\mathclap{it_*\, \in\, \sigma(JA)\, \cap\, i(0, +\infty)}}\ \dim \H_{t_*}.
			\end{equation}
		By Equation~\eqref{eq:sfisolcross} and by the concatenation axiom defining the spectral flow, we get
			\begin{equation}\label{eq:10meno2}
				\spfl \bigl( D, [\eps,T] \bigr) \equiv \quad\ \sum_{\mathclap{it_*\, \in\, \sigma(JA)\, \cap\, i(0, +\infty)}}\ \dim \H_{t_*} \quad \mod 2,
			\end{equation}
		and comparing \eqref{eq:9piu2} and \eqref{eq:10meno2} we infer
			\begin{equation}\label{eq:102}
				\spfl \bigl( D, [\eps,T] \bigr) \equiv - \frac{\dim \H_0}{2} \equiv \frac{\dim \H_0}{2} \quad \mod 2.
			\end{equation}
		Equations~\eqref{eq:finalespfl2} and \eqref{eq:92} also yield
			\begin{equation}\label{eq:112}
				\spfl \bigl( D, [\eps,T] \bigr) \equiv n^-(A) + \frac{\dim \H_0}{2} \quad \mod 2,
			\end{equation}
		and from the last two congruences \eqref{eq:102} and \eqref{eq:112}, we finally conclude that
			\[
				n^-(A) \equiv 0 \quad \mod 2.
			\]
		
		In the case where $A$ is invertible, the initial instant $t = 0$ is not a crossing and therefore we can repeat the previous discussion in a simpler way, by considering the spectral flow
		directly on the interval $[0, T]$ (cf.~Corollary~\ref{cor:MainA}).
	 \end{proof}
	 
	The following corollary is a direct consequence of Theorem~\ref{prop:mainB}; however, since the case is much simpler and does not require in fact the partial signatures, we give an
	independent proof. In this special case in which the matrix $JA$ is diagonalisable the result can be proved directly by arguing as in Proposition~\ref{thm:cruciale} and by taking into account the
	local contribution to the spectral flow as discussed in Lemma~\ref{thm:lemmafava}.
 	  
	\begin{cor}\label{cor:mainB2}
		If $A$ is invertible and $JA$ is linearly stable, then $n^-(A)$ is even.
	\end{cor}
	
	\begin{proof}
		First we observe that the second assumption implies that there is a bijection between the crossing instants $t_*$ and the pure imaginary eigenvalues of $JA$ of the form $it_*$ for positive
		real $t_*$. Let us then compute the crossing form $\Gamma(D, t_*)$ in correspondence of a crossing $t_* \in (0, +\infty)$: by definition it is given by
			\[
				\Gamma(D, t_*) \= Q\dot{D}(t_*)Q \bigr|_{\ker D(t_*)} = QGQ\bigr|_{\ker D(t_*)},
			\]
		where $Q$ is the orthogonal projection onto the kernel of $D(t_*)$. Note that the linear map $\Gamma(D, t_*)$ coincides (in the sense of Remark~\ref{rmk:crossform}) with the quadratic
		Krein form:
			\[
				\Gamma (D, t_*)[u] = \langle Gu, u \rangle, \qquad \forall\, u \in E_{it_*}(JA),
			\]
		since $\ker D(t_*) = E_{it_*}(JA)$ for every crossing $t_*$. By Krein theory and by the fact that $JA$ is diagonalisable, for any crossing instant $t_* \in (0,+\infty)$ the Krein form
		$g(u, u) \= \langle Gu,u \rangle$ is non-degenerate on each eigenspace $E_{it_*}(JA)$ and by Proposition~\ref{prop:MainA} there exists $T > 0$ such that
		$\spfl\bigl( D, [0, T_1] \bigr) = \spfl \bigl( D, [0, T_2] \bigr)$ for every $T_1, T_2 \geq T$. Thus we get
			\[
				\begin{split}
					\spfl \bigl( D, [0, T] \bigr) & = n^-(A) - n^- \bigl( D(T) \bigr) \\
							& = n^-(A) - 2n.
				\end{split}
			\]
		Since $JA$ is diagonalisable we have
			\[
				4n = \quad 2 \sum_{\mathclap{\lambda\, \in\, \sigma(JA)\, \cap\, i(0, +\infty)}}\ \dim E_\lambda,
			\]
		or, which is the same,
			\[
				2n = \quad \sum_{\mathclap{\lambda\, \in\, \sigma(JA)\, \cap\, i(0, +\infty)}}\ \dim E_\lambda
			\]
		Equation~\eqref{eq:utile} applied to the path $D$ yields
			\[
				\spfl \bigl( D, [0,T] \bigr) \equiv \quad  \sum_{\mathclap{\lambda\, \in\, \sigma(JA)\, \cap\, i(0, +\infty)}}\ \dim E_\lambda \quad \mod 2
			\]
		and we conclude that	
			\[
				n^-(A) \equiv 0 \quad \mod 2. \qedhere
			\]
	 \end{proof}
	 
	 \begin{rmk}
	 	We observe that Corollary~\ref{cor:mainB2} can be proved without using the technique of partial signatures also in the case where $A$ is not invertible.
		In order to take care of the crossing instant $t = 0$ it is enough to argue as in the proof of Theorem~\ref{prop:mainB}, with the only difference that, assuming diagonalisability, $\H_0$
		coincides with the kernel of $A$ (and, consequently, the kernel of $JA$).
	 \end{rmk}

	\section{Main theorem} \label{sec:main}

	We state and prove here the main result of our research, concerning the relationship between the Morse index of a critical point and the spectral instability of an associated relative
	equilibrium.
	
	Consider the matrix $B$ defined in \eqref{eq:generalB} and set
		\begin{equation} \label{eq:NsimB}
			N \= \begin{pmatrix}
				I & M\Xi \\
				0 & I
			\end{pmatrix}
			\begin{pmatrix}
				- D^2\mathcal{U}(\bar{x}) & \trasp{\Xi} \\
				\Xi & M^{-1}
			\end{pmatrix}
			\begin{pmatrix}
				I & 0 \\
				M\trasp{\Xi} & I
			\end{pmatrix} =
			\begin{pmatrix}
				-\big( D^2\mathcal{U}(\bar{x}) + \omega^2 M \big) & 0 \\
				0 & M^{-1}
			\end{pmatrix}.
		\end{equation}
	Observe that $D^2\mathcal{U}(\bar{x}) + \omega^2 M$ is precisely the Hessian $D^2\mathcal{U}_\Xi(\bar{x})$ of the augmented potential $\mathcal{U}_\Xi$ evaluated at its critical point
	$\bar{x}$ and define then the nullity and the Morse index of $\bar{x}$ as:
		\begin{align*}
			\nu(\bar{x}) & \= \nu \bigl( D^2\mathcal{U}_\Xi(\bar{x}) \bigr), \\
			\iMor(\bar{x}) & \= \iMor \bigl( D^2\mathcal{U}_\Xi(\bar{x}) \bigr).
		\end{align*}
	Thus we have the following theorem.
	
	\begin{thm} \label{thm:mainequiv}
		Let $\bar{x}$ be a critical point of the augmented potential function $\mathcal{U}_\Xi$ defined in \eqref{eq:augpotential} and assume that $\nu(\bar{x})$ is even.
		If $\iMor(\bar{x})$ is odd, then the relative equilibrium corresponding to $\bar{x}$ is spectrally unstable. 
	\end{thm}
	
	\begin{proof}
		Let $\H \= \C^{4n}$ and define the path $D : [0, +\infty) \to \Bsa(\H)$ as
			\[
				D(t) \= B + tG
			\]
		with $G \= iJ$, as in the previous section. We prove the contrapositive of the statement, that is, we show that if the relative equilibrium corresponding to the given critical point
		$\bar{x}$ is spectrally stable then its Morse index $\iMor(\bar{x})$ is even. Thus, assuming spectral stability, Theorem~\ref{prop:mainB} immediately yields
			\[
				n^-(B) \equiv 0 \quad \mod 2.
			\]
		Now, by Sylvester's Law of Inertia, we observe that
			\[
				n^-(B) = n^-(N),
			\]
		where $N$ is given by \eqref{eq:NsimB}, and since $n^-(N) = 2n - \iMor(\bar{x}) - \nu(\bar{x})$, it directly follows that
			\[
				\iMor(\bar{x}) \equiv 0 \quad \mod 2. \qedhere
			\]
	\end{proof}
	
	The next corollary is an immediate consequence of the previous theorem.
	
	\begin{cor} \label{cor:mainequiv}
		Let $\bar{x}$ be a critical point of the augmented potential function $\mathcal{U}_\Xi$. If $\iMor(\bar{x})$ or $\nu(\bar{x})$ are odd then the corresponding relative equilibrium is linearly
		unstable.
	\end{cor}
	
	\begin{rmk}
		Assuming linear stability we have that $\nu(\bar{x}) = \nu(JB)$, which is even due to the diagonalisability of $JB$.
	\end{rmk}

	\section{An important application: $n$-body-type problems} \label{sec:nbody}

	With reference to the notation and the setting outlined in the beginning of Section~\ref{sec:description}, we define two $n$-body-type problems by specifying two potential functions as follows.
	For each pair of indices $i, j \in \{1, \dots, n\}$, $i \neq j$, we let $\Delta_{ij}$ denote the \emph{collision set of the $i$-th and $j$-th particles}
		\[
			\Delta_{ij} \= \Set{ q \in \R^{2n} | q_i = q_j };
		\]
	we call $\displaystyle \Delta \= \bigcup_{i, j = 1}^n \Delta_{ij}$ the \emph{collision set} (by definition, then, $\Delta$ is a union of hyperplanes) and $X \= \R^{2n} \setminus \Delta$ the
	\emph{(collision-free) configuration space}.

	On this set (which is a cone in $\R^{2n}$) we define the \emph{potential functions} $\Uhom, \Ulog: X \to \R$ (generally denoted by $U$) as
		\begin{subequations} \label{eq:potentials}
			\begin{align}
				\Uhom(q) &\= \sum_{\substack{i, j = 1\\ i<j}}^n \frac{m_i m_j}{\abs{q_i - q_j}^{\alpha}}, \qquad \alpha \in (0, 2); \label{eq:Uhom} \\
				\Ulog(q) &\= -\sum_{\substack{i, j = 1\\ i<j}}^n m_i m_j\,\log\abs{q_i - q_j}. \label{eq:Ulog}
			\end{align}
		\end{subequations}
	From now on, unless otherwise specified, every reference to the contents of Section~\ref{sec:description} will be intended as concerning these two potential, \ie we consider $\mathcal{U} = U$.
	
	\begin{rmk}
		Note that for $\alpha = 1$ one finds the gravitational potential of the classical $n$\nobreakdash-body problem. Moreover, the logarithmic potential can be considered as a limit case of the
		$\alpha$\nobreakdash-ho\-mo\-ge\-ne\-ous%
			\footnote{The $\alpha$-homogeneous potential is actually homogeneous of degree $-\alpha$; however, we call it in this way for the sake of simplicity.}
		one, in the following sense:
			\[
				\frac{\Uhom(q) - 1}{\alpha} \sim \Ulog(q), \qquad \alpha \to 0^+,
			\]
		for every $q \in X$. Nevertheless, it displays quite a different behaviour with respect to $\Uhom$, as we shall show.
	\end{rmk}
	
	Since the centre of mass of the system moves with uniform rectilinear motion, without loss of generality we can fix it at the origin, that is we can set $\sum_{i = 1}^n m_i q_i = 0$.
	We thus consider the \emph{reduced (collision-free) configuration space} as follows:
		\[
 			\hat{X} \= \Set{ q \in X | \sum_{i = 1}^n m_i q_i = 0}.
		\]

	\begin{rmk}
		We observe that the Hamiltonian flow of System~\eqref{eq:Hamilton} is well defined on $T^*\hat{X}$ but it is not complete on $T^*\R^{2n}$, due to the existence of solutions
		for which the potential escapes to infinity in a finite time. This happens, for instance, for initial conditions leading to a collision between two or more particles.
	\end{rmk}

		\subsection{Central configurations and relative equilibria} \label{subsec:cc}
	
	We recall here some well-known facts about central configurations and fix our notation. For further references in the classical gravitational case, we	refer to \cite{Moeckel}.

	Let $a, b \in \R$, with $a < b$. We call $\bar{q} \in \hat{X}$ a \emph{(planar) central configuration} if there is some smooth real-valued function $r : (a, b) \to \R$, with $r(t) > 0$ for all
	$t \in (a, b)$, such that
		\begin{equation} \label{def:CC}
			q(t) \= r(t)\, \bar{q}
		\end{equation}
	is a (classical) solution of Newton's Equations~\eqref{eq:Newton}. Here $\bar{q}$ represents the constant shape of the configuration, while $r(t)$ its time-depending size. Substituting
	\eqref{def:CC} into \eqref{eq:Newton} we obtain:

	\begin{description}
		\item \emph{$\alpha$-homogeneous case:}
				\[
					\ddot{r} M \bar{q} = r^{-(\alpha+1)} \nabla \Uhom(\bar{q}). 
				\]
			Taking the scalar product with $\bar{q}$ in both sides of the above equality and applying Euler's theorem on homogeneous functions, we get
			$\ddot{r} = - \lambda_\alpha / r^{\alpha+1}$, where
				\begin{equation} \label{eq:lambdahom}
					\quad \lambda_\alpha \= \frac{\alpha \Uhom(\bar{q})}{\I(\bar{q})}.
				\end{equation}
		\item \emph{Logarithmic case:} 
				\[
					\ddot{r} M \bar{q} = r^{-1} \nabla \Ulog(\bar{q}).
				\]
			Taking again the scalar product with $\bar{q}$ as before, we get $\ddot{r} = -\lambda_{\log} / r$, where
				\[
					\quad \lambda_{\mathrm{log}} \= -\frac{\langle \nabla \Ulog(\bar{q}), \bar{q} \rangle}{\I(\bar{q})}.
				\]
			A straightforward computation shows that $\displaystyle - \langle \nabla \Ulog(\bar{q}), \bar{q} \rangle = \sum_{\substack{i, j = 1 \\ i < j}}^n m_i m_j \eq \mathcal{M}$, so that
				\begin{equation} \label{eq:lambdalog}
					\lambda_{\mathrm{log}} =  \dfrac{\mathcal{M}}{\I(\bar{q})}.
				\end{equation}
	\end{description}
	
	\begin{rmk} \label{rmk:lambdalog}
		It is worthwhile noting that in the logarithmic case the Lagrange multiplier depends only on the size of the central configuration (via the moment of inertia) and not on its shape.
	\end{rmk}
	
	In both cases, a central configuration $\bar{q}$ satisfies the \emph{central configurations equation}
		\begin{equation} \label{eq:cc}
			M^{-1} \nabla U(\bar{q}) + \lambda \bar{q} = 0,
		\end{equation}
	where $\lambda = \lambda_\alpha$ (resp.~$\lambda = \lambda_{\mathrm{log}}$) when $U = U_{\alpha}$ (resp.~$U=\Ulog$).
	Thus we can also look at a central configuration as a special distribution of the bodies in which the acceleration vector of each particle lines up with its position vector, and the
	proportionality constant $\lambda$ is the same for all particles. Equation~\eqref{eq:cc} is a quite complicated system of nonlinear algebraic equations and only few solutions are
	known.

	Let us now introduce the \emph{ellipsoid of inertia} (also called the \emph{standard ellipsoid})
		\[
			\S \= \Set{ q \in \hat{X} | \I(q) = 1 }.
		\]
	If $\bar{q}$ is a central configuration, then so are $c\bar{q}$ and $R \bar{q}$, for any $c \in \R \setminus \{ 0 \}$ and any $2n \times 2n$ block-diagonal matrix $R$ with blocks given by a
	$2 \times 2 $ fixed matrix in $\SO(2)$. We observe that the rescaled configuration $c\bar{q}$ solves a system analogous to \eqref{eq:cc} obtained by replacing $\lambda_\alpha$ with
	$\widetilde{\lambda}_\alpha \= \lambda_\alpha \abs{c}^{-(\alpha+2)}$ and $\lambda_{\log}$ with $\widetilde{\lambda}_{\log} \= \lambda_{\log}\abs{c}^{-2}$.
	Because of these facts, it is standard practice to count central configurations by fixing a constant $c$ (the \lq\lq scale\rq\rq: this actually means to work on $\S$) and to identify all those
	which are rotationally equivalent. This amounts to take the quotient of the configuration space $\hat{X}$ with respect to homotheties and rotations about the origin, or, which is the same, to
	consider the so-called \emph{shape sphere}
		\[
			\mathbb{S} \= \S / \SO(2).
		\]
	Note that the second equation of System~\eqref{eq:equilibrium2} (with $\mathcal{U} = U$) is precisely the Central Configurations Equation~\eqref{eq:cc}, with the square modulus of the
	angular velocity as Lagrange multiplier. Intuitively speaking, then, if we let $n$ bodies, distributed in a planar central configuration, rotate with an angular velocity $\omega$ equal to
	$\sqrt{\lambda_\alpha}$ or $\sqrt{\lambda_{\log}}$ (depending on the potential they are subject to), we get a relative equilibrium, which becomes an equilibrium in a uniformly rotating
	coordinate system.
				
	Motivated by the observation that one can write, for every $q \in \hat{X}$,
		\[
			U(q) = U \left( \sqrt{\I(q)} \frac{q}{\sqrt{\I(q)}} \right) = \begin{cases}
									\I^{-\frac{\alpha}{2}}(q)\, \Uhom \left( \dfrac{q}{\sqrt{\I(q)}} \right) &\text{if $U = \Uhom$} \\ \\
									\Ulog \left( \dfrac{q}{\sqrt{\I(q)}} \right) - \dfrac{\M}{2} \log \I(q) &\text{if $U = \Ulog$}
									\end{cases}
		\]	
	we define, as in \cite{BarSec}, the maps $f_\alpha, f_{\log} : \hat{X} \to \R$ respectively as
		\[
			f_\alpha (q) \= \I^{\frac{\alpha}{2}}(q) \Uhom(q) \qquad \text{and} \qquad f_{\mathrm{log}}(q) \= \Ulog(q) + \frac{\mathcal{M}}{2} \log \I(q),
		\]
	so that, restricting to the ellipsoid of inertia $\S$, we have
		\[
 			f_\alpha (q) = \Uhom \bigr|_{\S}(q)	\qquad \text{and} \qquad		f_{\mathrm{log}}(q) = \Ulog \bigr|_{\mathcal S}(q), \qquad \forall q \in \S.
		\]
	The reason for introducing these functions lies in the fact that we want to find the critical points of the potentials $\Uhom, \Ulog$ \emph{constrained to $\S$}: we shall now show that it is
	possible to compute them more easily as \emph{free} critical points of $f_\alpha$ and $f_{\log}$. Since the manifold $\S$ is topologically a sphere, we can avoid the use of the
	covariant derivative for this purpose.
	
	For every $(q,v) \in T \hat X$ we calculate, in the standard basis of $\R^{2n}$,
		\begin{subequations}\label{eq:grad}
			\begin{align}
				\langle \nabla f_\alpha (q),v \rangle &= \frac{\alpha}{2} \I^{\frac{\alpha}{2} - 1}(q) \Uhom(q) \langle \nabla \I(q),v \rangle + \I^\frac{\alpha}{2}(q) \langle \nabla\Uhom(q),v \rangle, \\
				\langle \nabla f_{\mathrm{log}}(q), v \rangle &= \langle \nabla\Ulog(q),v \rangle + \frac{\M}{2\I(q)} \langle \nabla \I(q),v \rangle.
			\end{align}
		\end{subequations}
	Now, recalling that $\I(q) = 1$ on $\mathcal{S}$ and that $\langle \nabla \I(q), v \rangle = \langle 2Mq, v \rangle$, we obtain, for every $q \in \S$ and every $v \in T_q\hat X$:
		\begin{subequations}\label{eq:gradU}
			\begin{align}
				\langle \nabla \Uhom \bigr|_{\S} (q), v \rangle &= \langle \nabla \Uhom(q), v \rangle + \alpha \Uhom(q) \langle Mq, v \rangle, \\
				\langle \nabla \Ulog \bigr|_{\S}(q), v \rangle &= \langle \nabla \Ulog(q), v \rangle + \M \langle Mq, v \rangle. 
			\end{align}
		\end{subequations}
	It is now clear, comparing Equations~\eqref{eq:cc} and \eqref{eq:gradU} and using \eqref{eq:lambdahom} and \eqref{eq:lambdalog}, that the constrained critical points of the restricted
	potentials $\Uhom|_{\S}$ and $\Ulog|_{\S}$ are precisely the central configurations.

	From Equations~\eqref{eq:grad} we compute the Hessians of $f_\alpha$ and $f_{\mathrm{log}}$ for every $(q, v) \in T \hat{X}$:
		\begin{subequations}\label{eq:Hess}
			\begin{align}
				\begin{split}
					\langle D^2 f_\alpha (q)v, v \rangle &= \frac{\alpha}{2} \left( \frac{\alpha}{2} - 1 \right) \I^{\frac{\alpha}{2} - 2}(q) \Uhom(q) \langle \nabla \I(q), v \rangle^2
													+ \alpha \I^{\frac{\alpha}{2} - 1}(q) \langle \nabla \Uhom(q), v \rangle \langle \nabla \I(q), v \rangle \\
												& \quad\, + \frac{\alpha}{2} \I^{\frac{\alpha}{2} - 1}(q) \Uhom(q) \langle D^2 \I(q)v, v \rangle
													+ \I^{\frac{\alpha}{2}}(q) \langle D^2\Uhom(q)v, v \rangle,
				\end{split} \raisetag{18.1pt} \\
				\langle D^2 f_{\mathrm{log}}(q)v, v \rangle &= \langle D^2 \Ulog(q)v, v \rangle + \frac{\mathcal{M}}{2}\left( -\frac{\langle \nabla \I(q),v \rangle^2}{\I^2(q)}
													+ \frac{\langle D^2 \I(q)v, v \rangle}{\I(q)} \right).
			\end{align}
		\end{subequations}
	Assuming that $q \in \mathcal{S}$ is a central configuration for $\Uhom$ (resp.~for $U_{\mathrm{log}}$) and recalling that $\langle D^2 \I(q)v, v \rangle = \langle 2Mv, v \rangle$, from
	\eqref{eq:cc} and \eqref{eq:Hess} we obtain, for every $v \in T_q\hat{X}$:
		\[
			\begin{split}
				\langle D^2 \Uhom \bigr|_{\S}(q) v, v \rangle &= \langle D^2 \Uhom(q)v, v \rangle + \alpha \Uhom(q) \langle Mv, v \rangle
					- \alpha (\alpha + 2) \Uhom(q) \langle Mq, v \rangle^2, \\
				\langle D^2 \Ulog \bigr|_{\S}(q) v, v \rangle &= \langle D^2 \Ulog(q)v, v \rangle + \M \big\{ \langle Mv,v \rangle - \langle Mq,v \rangle^2 \big\}.
			\end{split}
		\]
	Choosing $v \in T_q \mathcal{S}$, these last expressions can be simplified, since the equality $\langle Mq, v \rangle = 0$ holds:
		\begin{subequations} \label{eq:hessristr}
			\begin{align}
					\langle D^2 \Uhom \bigr|_{\S}(q) v, v \rangle &= \langle D^2 \Uhom(q)v, v \rangle + \alpha \Uhom(q) \langle Mv, v \rangle, \\
					\langle D^2 \Ulog \bigr|_{\S}(q) v, v \rangle &= \langle D^2 \Ulog(q)v, v \rangle + \M \langle Mv, v \rangle.
			\end{align}
		\end{subequations}
	Thus, for any central configuration $q \in \mathcal S$, the previous equations ensure that the Hessians of the restrictions of $\Uhom$ and $\Ulog$ to $\S$ are restrictions to $T_q\S$ of
	quadratic forms defined on the whole $T_q \hat{X}$.
	
	\begin{rmk}
		The previous equations still hold unchanged also if we restrict the potentials to the shape sphere $\mathbb{S}$.
	\end{rmk}

		\subsection{A symplectic decomposition of the phase space} \label{subsec:decomp}
	
	We continue our analysis by presenting here a symplectic splitting of the phase space which reflects the invariance of the $n$-body-type problems under some isometries.
	There are three components: the first one, denoted by $E_1$, represents the translational invariance, $E_2$ is the subspace generated by all rotations and dilations of the central
	configuration and the third one, $E_3$, is the symplectic complement of the other two.
	The reason behind this construction is that, due to the existence of the first integrals, there are eight eigenvalues of the linearised matrix which are
	always present, independently of the number of bodies $n$: accordingly, we isolate them and focus only on the remaining $4n - 8$, the ones holding the heart of the dynamics.
		
	When linearising around a relative equilibrium $\bar{\zeta}$, the $4n \times 4n$ Hamiltonian matrix associated to System~\eqref{eq:hamsyslinearized} is
		\begin{equation} \label{eq:L}
			L \= -JB = \begin{pmatrix}
						\omega K & M^{-1} \\
						D^2 U(\bar{x}) & \omega K
					\end{pmatrix},
		\end{equation}
	where, we recall, each block is a square matrix of size $2n \times 2n$. Since it will be necessary, in the following, to know the explicit expressions of the Hessians of the two potentials
	$\Uhom$ and $\Ulog$, we write them down here:
		\begin{subequations} \label{freehessianexpr}
			\begin{align}
				& D^2 \Uhom(x) \eq \bigl( S^{(\alpha)}_{ij} \bigr), &\quad &\text{with } \left\{ \begin{aligned}
											& S^{(\alpha)}_{ij} \= \alpha \frac{m_i m_j}{\abs{x_i - x_j}^{\alpha + 2}}\ \big[ I_2 - (\alpha + 2)u_{ij}\trasp{u}_{ij} \big] \qquad \text{if } j \neq i \\ 
											& S^{(\alpha)}_{ii} \= -\sum_{\substack{j = 1 \\ j \neq i}}^n S^{(\alpha)}_{ij}
										\end{aligned}\right. \raisetag{1.45cm} \\
				\notag \\
				& D^2 \Ulog(x) \eq \bigl( S^{(\log)}_{ij} \bigr), &\quad &\text{with } \left\{ \begin{aligned}
														& S^{(\log)}_{ij} \= \frac{m_i m_j}{\abs{x_i - x_j}^2} \big[ I_2 - 2 u_{ij} \trasp{u}_{ij} \big] \qquad \text{if } j \neq i \\
														& S^{(\log)}_{ii} \= -\sum_{\substack{j = 1 \\ j \neq i}}^n S^{(\log)}_{ij}
													\end{aligned}\right.
			\end{align}
		\end{subequations}
	where $u_{ij} \= \dfrac{x_i - x_j}{\abs{x_i - x_j}}$ and the indices $i$ and $j$ vary in $\{ 1, \dotsc, n \}$.

	Going back to the linearisation, we note that the first integrals of motion and the symmetries of the system generate two linear symplectic subspaces of the phase space
	$T^*X \cong X \times \R^{2n}$ which are invariant under $L$. Indeed, a basis for the position and momentum of the centre of mass is given by the four vectors in $\R^{4n}$
		\[
			v_1 \= \begin{pmatrix}
					v \\
					0
				\end{pmatrix}, \qquad
			v_2 \= \begin{pmatrix}
					Kv \\
					0
				\end{pmatrix}, \qquad
			v_3 \= \begin{pmatrix}
					0 \\
				 	Mv
				\end{pmatrix}, \qquad
			v_4 \= \begin{pmatrix}
					0 \\
					KMv
				\end{pmatrix}
		\]
	with $v \= \trasp{(1, 0, 1, 0, \dots, 1, 0)} \in \R^{2n}$. If we let $E_1$ denote the space spanned by these vectors, with the following computations we see that it is $L$-invariant:
		\begin{align*}
			Lv_1 & = \begin{pmatrix}
					\omega K v \\
					D^2 U(\bar{x}) v
				\end{pmatrix}
				= \omega v_2 + \begin{pmatrix}
								0 \\
								D^2 U(\bar{x}) v
							\end{pmatrix}
				= \omega v_2, \\
			Lv_2 & = \begin{pmatrix}
					\omega K^2 v \\
					D^2 U(\bar{x}) K v
				\end{pmatrix}
				= - \omega v_1 + \begin{pmatrix}
							0 \\
							D^2 U(\bar{x}) Kv
							\end{pmatrix}
				= - \omega v_1,\\
			Lv_3 & = 	\begin{pmatrix}
					M^{-1} M v \\
					\omega K M v
				\end{pmatrix}
				=  v_1 + \begin{pmatrix}
							0 \\
							\omega KMv
							\end{pmatrix}
				= v_1 + \omega v_4, \\
			Lv_4 & = 	\begin{pmatrix}
					M^{-1} KMv \\
					\omega K^2 Mv
				\end{pmatrix}
				=  v_2 + \begin{pmatrix}
							0 \\
							\omega K^2 Mv
							\end{pmatrix}
				= v_2 - \omega v_3,
		\end{align*}
	since $K$ and $M$ commute and $D^2 U(\bar{x}) v = D^2 U(\bar{x}) Kv = 0$ for both $\Uhom$ and $\Ulog$, due to their matrix structure \eqref{freehessianexpr}.
	The invariant space $E_1$ is also symplectic, because the standard symplectic form%
		\footnote{We take as standard symplectic form $\Omega$ on $\R^{4n}$ that one induced by the complex structure $J_{4n}$ (denoted again by $J$).}
	$\Omega_1 \= \Omega|_{E_1 \times E_1}$ of $(\R^{4n}, \Omega)$ restricted to $E_1$ is non-degenerate: we have indeed that $\Omega_1(v_1, v_3) = \langle J v_1, v_3 \rangle =
	\trasp{v}Mv \neq 0$ and $\Omega_1(v_2, v_4) = \langle J v_2, v_4 \rangle = \trasp{v}Mv \neq 0$. We denote by $L_1$ the restriction $L|_{E_1}$ of $L$ to $E_1$; from the calculations
	performed above to show the invariance of $E_1$, it follows that it is given, in the basis $(v_1, v_2, w_1, w_2)$, by the $4 \times 4$ matrix
		\[
			L_1 \= \begin{pmatrix}
					0 & -\omega & 1 & 0 \\
					\omega & 0 & 0 & 1 \\
					0 & 0 & 0 & -\omega \\
					0 & 0 & \omega & 0
				\end{pmatrix}.
		\]
	Its eigenvalues are $\pm i \omega$, each with algebraic multiplicity $2$; however, the dimension of the associated eigenspaces is $1$, and therefore $L_1$ is not diagonalisable.
	Note that the symplectic complement $E_1^{\perp_\Omega}$ of $E_1$ is the space where the centre of mass of the system is fixed at the origin and the total linear momentum is zero.

	The scaling symmetry and the conservation of the angular momentum generate another linear symplectic $L$-invariant subspace $E_2$, a basis of which is given by the four vectors
	in $\R^{4n}$
		\[
			w_1 \= \begin{pmatrix}
					\bar{x} \\
					0
				\end{pmatrix}, \qquad
			w_2 \= \begin{pmatrix}
					K\bar{x} \\
					0
				\end{pmatrix}, \qquad
			w_3 \= \begin{pmatrix}
					0 \\
					M\bar{x}
				\end{pmatrix}, \qquad
			w_4 \= \begin{pmatrix}
					0 \\
					KM\bar{x}
				\end{pmatrix}
		\]
	To show that this is $L$-invariant, we compute:
		\begin{align*}
			Lw_1 & = \begin{pmatrix}
					\omega K \bar{x} \\
					D^2 U(\bar{x}) \bar{x}
				\end{pmatrix}
				= \omega w_2 + \begin{pmatrix}
								0 \\
								D^2 U(\bar{x}) \bar{x}
							\end{pmatrix}
				= \begin{cases}
					\omega w_2 + (\alpha + 1)\omega^2 w_3 & \text{if $U = \Uhom$} \\
					\omega w_2 + \omega^2 w_3 & \text{if $U = \Ulog$}
				   \end{cases}, \\
			Lw_2 & = \begin{pmatrix}
					\omega K^2 \bar{x} \\
					D^2 U(\bar{x}) K\bar{x}
				\end{pmatrix}
				= - \omega w_1 + \begin{pmatrix}
							0 \\
							D^2 U(\bar{x}) K\bar{x}
							\end{pmatrix}
				= -\omega w_1 - \omega^2 w_4, \\
			Lw_3 & = 	\begin{pmatrix}
					M^{-1} M\bar{x} \\
					\omega KM\bar{x}
				\end{pmatrix}
				=  w_1 + \begin{pmatrix}
							0 \\
							\omega KM\bar{x}
							\end{pmatrix}
				= w_1 + \omega w_4, \\
			Lw_4 & = 	\begin{pmatrix}
					M^{-1} KM\bar{x} \\
					\omega K^2 M\bar{x}
				\end{pmatrix}
				=  w_2 + \begin{pmatrix}
							0 \\
							\omega K^2 M\bar{x}
							\end{pmatrix}
				= w_2 - \omega w_3.
		\end{align*}
	The first relation is obtained from Euler's theorem on homogeneous functions applied to $\nabla U(x)$:
		\[
			D^2 U(x) x = D\big( \nabla U(x) \big) x = \begin{cases}
										-(\alpha + 1) \nabla\Uhom(x) & \text{if $U = \Uhom$} \\
										- \nabla\Ulog(x) & \text{if $U = \Ulog$}
									\end{cases}
		\]
	and using the central configurations equation for relative equilibria. The second one comes from the invariance of the potentials under rotations: following \cite{Moeckel}, we have indeed that
	$U\big( R(t)x \big) = U(x)$ for every $R(t) \= e^{\omega Kt}$, and differentiating this relation with respect to $x$ we obtain%
		\footnote{We denote by $DU(x)$ the transpose of $\nabla U(x)$.}
	$DU \big( R(t)x \big) R(t) = DU(x)$. If we differentiate again with respect to $t$ at $t = 0$ and divide by $\omega$, we get
		\[
			\trasp{(Kx)}D^2 U(x) + DU(x)K = 0.
		\]
	When $x = \bar{x}$ is a central configuration associated with a relative equilibrium, as it is in this case, $DU(\bar{x}) = -\omega^2 \trasp{(M\bar{x})}$ and the equation above becomes, dividing
	both sides by $\omega$:
		\[
			\trasp{(K\bar{x})}D^2U(\bar{x}) - \omega^2 \trasp{(M\bar{x})}K = 0,
		\]
	or, equivalently,
		\[
			\trasp{(K\bar{x})}D^2U(\bar{x}) + \omega^2 \trasp{(KM\bar{x})} = 0.
		\]
	Because of the symmetry of the Hessian it is now sufficient to take the transpose of both sides to conclude.
	
	The space $E_2$ is again symplectic: a verification of the non-degeneracy of $\Omega_2 \= \Omega|_{E_2 \times E_2}$ is completely analogous to the one that we performed above for
	$E_1$. The matrices of $L|_{E_2}$ with respect to the basis $(w_1, w_2, w_3, w_4)$ are
		\begin{gather*}
			L_2^{(\alpha)} \= \begin{pmatrix}
						0 & -\omega & 1 & 0 \\
						\omega & 0 & 0 & 1 \\
						(\alpha + 1)\omega^2 & 0 & 0 & -\omega \\
						0 & -\omega^2 & \omega & 0
					\end{pmatrix} \qquad \text{if $U = \Uhom$} \\
	\intertext{and}
			L_2^{(\log)} \= \begin{pmatrix}
						0 & -\omega & 1 & 0 \\
						\omega & 0 & 0 & 1 \\
						\omega^2 & 0 & 0 & -\omega \\
						0 & -\omega^2 & \omega & 0
					\end{pmatrix} \qquad \text{if $U = \Ulog$}.
		\end{gather*}
	Their eigenvalues are $0$ (with algebraic multiplicity $2$) and $\pm i\omega\sqrt{2 - \alpha}$ in the homogeneous case ($\Uhom$), and $0$ (with algebraic multiplicity $2$) and
	$\pm i\omega\sqrt{2}$ in the logarithmic case ($\Ulog$). Again, these matrices are not diagonalisable because the eigenspace associated to $0$ is only one-dimensional.
	Table~\vref{tab:eigenvalues} summarises the information obtained thus far about these first eight eigenvalues.
		\begin{table}[tb]
			\caption{Eigenvalues of $L_1$ and $L_2$ for both potentials.}
			\label{tab:eigenvalues}
			\centering
			\begin{tabular}{cccc}
				\toprule
				 Potential	& \multicolumn{2}{c}{Eigenvalue} & Multiplicity \\
				 \midrule
				 \multirow{5}{*}{$\Uhom$}	& \multirow{2}{*}{$L_1^{(\alpha)}$} & $i\omega$ & 2 \\
					&  & $-i\omega$ & 2 \\
					\cmidrule(rl){2-4}
					& \multirow{3}{*}{$L_2^{(\alpha)}$} & $0$ & 2 \\
					&  & $i\omega\sqrt{2 - \alpha}$ & 1 \\
					&  & $-i\omega\sqrt{2 - \alpha}$ & 1 \\
				\cmidrule(rl){1-4}
				\multirow{5}{*}{$\Ulog$} & \multirow{2}{*}{$L_1^{(\log)}$} & $i\omega$ & 2 \\
					&  & $-i\omega$ & 2 \\
					\cmidrule(rl){2-4}
					& \multirow{3}{*}{$L_2^{(\log)}$} & $0$ & 2 \\
					&  & $i\omega\sqrt{2}$ & 1 \\
					&  & $-i\omega\sqrt{2}$ & 1 \\
				 \bottomrule
			\end{tabular}
		\end{table}
		      
	Thus, a relative equilibrium is always degenerate and not linearly stable in the classical sense. For this reason, we shall consider the restriction $L_3 \= L|_{E_3}$ of $L$ to the
	skew-orthogonal complement
		\begin{equation} \label{eq:E3}
			E_3 \= (E_1 \oplus E_2)^{\perp_\Omega},
		\end{equation}
	which is a linear symplectic subspace of dimension $4n - 8$ of $\R^{4n}$. Following \cite{Moeckel}, we adopt the following terminology.
	
	\begin{defn}
		A relative equilibrium is \emph{non-degenerate} if the remaining $4n - 8$ eigenvalues (relative to $L_3$) are different from $0$; we say that it is \emph{spectrally stable} if these
		eigenvalues are pure imaginary and \emph{linearly stable} if, in addition to this condition of spectral stability, $L_3$ is diagonalisable.
	\end{defn}
	
	In order to understand the structure of $L_3$, let us now consider the following change of variables:
		\[
			\begin{cases}
				x \mapsto C\xi \\
				\trasp{y} \mapsto \trasp{(C^{-1})}\trasp{\eta},
			\end{cases}
		\]
	where $C$ is a $2n \times 2n$ invertible matrix such that $[C, K] = 0$ and $\trasp{C}MC = I$. Then we have, for every $(x, y) \in T^*X$:
		\[
			L \begin{pmatrix}
					x \\
					\trasp{y}
				\end{pmatrix} = \begin{pmatrix}
								\omega K & M^{-1} \\
								D^2 U(\bar{x}) & \omega K
							\end{pmatrix}
							\begin{pmatrix}
								C\xi \\
								\trasp{(C^{-1})} \trasp{\eta}
							\end{pmatrix} = \begin{pmatrix}
											\omega KC\xi + M^{-1}\trasp{(C^{-1})} \trasp{\eta} \\
											D^2 U(\bar{x})C\xi + \omega K \trasp{(C^{-1})} \trasp{\eta}
										\end{pmatrix}.
		\]
	From the first condition on $C$ we find that also $\trasp{(C^{-1})}$ commutes with $K$, while from the second one we get that $\trasp{(C^{-1})} = MC$, so that we can write
		\[
			L \begin{pmatrix}
					x \\
					\trasp{y}
				\end{pmatrix} = \begin{pmatrix}
								C(\omega K\xi + \trasp{\eta}) \\
								\trasp{(C^{-1})} \big( \trasp{C} D^2 U(\bar{x})C\xi + \omega K \trasp{\eta} \big)
							\end{pmatrix} = \begin{pmatrix}
											C & 0 \\
											0 & \trasp{(C^{-1})}
										\end{pmatrix} \begin{pmatrix}
														\omega K & I \\
														\trasp{C} D^2 U(\bar{x}) C & \omega K
													\end{pmatrix} \begin{pmatrix}
																	\xi \\
																	\trasp{\eta}
																\end{pmatrix}.
		\]
	The matrix $C$ can be thought of as made up of $2 \times 2$ blocks of the form $(b, Jb)$, for any vector $b \in \R^2$; furthermore, it can be shown (see \cite{MeySch}) that, using a
	Gram-Schmidt-type algorithm, the first four columns of $C$ can be chosen as $(v, Kv, \bar{x}, K\bar{x})$, where $v$ is, as before, the vector $\trasp{(1,0, \dotsc, 1, 0)} \in \R^{2n}$.
	Looking now at the structures of (the first columns of) $C$ and $K$, one can recover the restrictions $L_1$ and $L_2$ from the equation above and derive the expression for the
	$(4n - 8) \times (4n - 8)$ matrix representing $L_3$:
		\[
			L_3 \= \begin{pmatrix}
					\omega K & I \\
					\D & \omega K
				\end{pmatrix},
		\]
	where every block has dimension $(2n - 4) \times (2n - 4)$ and $\D$ is the Hessian $\trasp{C} D^2 U(\bar{x}) C$ restricted to $E_3$, acting on the last $2n - 4$ components of $\xi$.
	The study of the linear stability of the relative equilibrium $\bar{z}$ amounts then to determine whether or not this matrix is spectrally stable and/or diagonalisable.

		\subsection{An example: the equilateral triangle}
	
	It is easy to see that the Lagrangian triangle with equal masses is a central configuration both for the $\alpha$-homogeneous potential and the logarithmic one. Indeed, both of them give rise
	to a central force field and the symmetry of a regular polygon is a sufficient condition for the bodies to satisfy Equation~\eqref{eq:cc}. We analyse here the behaviour of this relative equilibrium
	with respect to linear stability for both potentials.
	
	For simplicity of computation we set
		\[
			m_1 \= m_2 \= m_3 \= 1
		\]
	in both situations. The centre of mass is fixed at the origin and the setting is as described previously, specially Subsection~\ref{subsec:releq}.
	
	In the $\alpha$-homogeneous case we have that $\omega = \sqrt{3\alpha}$, hence the matrix $L^{(\alpha)}$ of the linearised problem is
		\[
			L^{(\alpha)} = \left( \begin{smallmatrix}
							0 & -\sqrt{3\alpha} & 0 & 0 & 0 & 0 & 1 & 0 & 0 & 0 & 0 & 0 \\
							\sqrt{3\alpha} & 0 & 0 & 0 & 0 & 0 & 0 & 1 & 0 & 0 & 0 & 0 \\
							0 & 0 & 0 & -\sqrt{3\alpha} & 0 & 0 & 0 & 0 & 1 & 0 & 0 & 0 \\
							0 & 0 & \sqrt{3\alpha} & 0 & 0 & 0 & 0 & 0 & 0 & 1 & 0 & 0 \\
							0 & 0 & 0 & 0 & 0 & -\sqrt{3\alpha} & 0 & 0 & 0 & 0 & 1 & 0 \\
							0 & 0 & 0 & 0 & \sqrt{3\alpha} & 0 & 0 & 0 & 0 & 0 & 0 & 1 \\
							\frac{1}{2}a & 0 & -\frac{1}{4}a & -\frac{\sqrt{3}}{4}b & -\frac{1}{4}a &
																													\frac{\sqrt{3}}{4}b &
																											0 & -\sqrt{3\alpha} & 0 & 0 & 0 & 0 \\
							0 & \frac{1}{2}c & -\frac{\sqrt{3}}{4}b & -\frac{1}{4}c & \frac{\sqrt{3}}{4}b &
																													-\frac{1}{4}c &
																											\sqrt{3\alpha} & 0 & 0 & 0 & 0 & 0 \\
							-\frac{1}{4}a & -\frac{\sqrt{3}}{4}b & \frac{1}{2}\alpha + \frac{5}{4}\alpha^2 & \frac{\sqrt{3}}{4}b &
																													-\alpha(\alpha + 1) & 0 &
																											0 & 0 & 0 & -\sqrt{3\alpha} & 0 & 0 \\
							-\frac{\sqrt{3}}{4}b & -\frac{1}{4}c & \frac{\sqrt{3}}{4}b & -\frac{1}{2}\alpha + \frac{3}{4}\alpha^2 & 0 & \alpha &
																											0 & 0 & \sqrt{3\alpha} & 0 & 0 & 0 \\
							-\frac{1}{4}a & \frac{\sqrt{3}}{4}b & -\alpha(\alpha + 1) & 0 & \frac{1}{2}\alpha + \frac{5}{4}\alpha^2 &
																													-\frac{\sqrt{3}}{4}b &
																											0 & 0 & 0 & 0 & 0 & -\sqrt{3\alpha} \\
							\frac{\sqrt{3}}{4}b & -\frac{1}{4}c & 0 & \alpha & -\frac{\sqrt{3}}{4}b & -\frac{1}{2}\alpha + \frac{3}{4}\alpha^2 &
																											0 & 0 & 0 & 0 & \sqrt{3\alpha} & 0
						\end{smallmatrix}
					\right),
		\]
	where $a \= \alpha(\alpha - 2)$, $b \= \alpha(\alpha + 2)$ and $c \= \alpha(3\alpha + 2)$. Its eigenvalues are
		\begin{align*}
			\lambda_1 & \= i\sqrt{3\alpha}, & \lambda_5 & \= 0,  & \lambda_9 & \= \frac{1}{2}\sqrt{6\alpha^2 + 12\alpha(i\sqrt{2\alpha} - 1)}, \\
			\lambda_2 & \= -i\sqrt{3\alpha}, & \lambda_6 & \= 0, & \lambda_{10} & \= -\frac{1}{2}\sqrt{6\alpha^2 + 12\alpha(i\sqrt{2\alpha} - 1)}, \\
			\lambda_3 & \= i\sqrt{3\alpha}, & \lambda_7 & \= i\sqrt{3\alpha(2 - \alpha)}, & \lambda_{11} & \= \frac{1}{2}\sqrt{6\alpha^2 - 12\alpha(i\sqrt{2\alpha} + 1)}, \\
			\lambda_4 & \= -i\sqrt{3\alpha}, & \lambda_8 & \= -i\sqrt{3\alpha(2 - \alpha)}, & \lambda_{12} & \= -\frac{1}{2}\sqrt{6\alpha^2 - 12\alpha(i\sqrt{2\alpha} + 1)}.
		\end{align*}
	The first four are those relative to the subspace $E_1$, the second four are related to the subspace $E_2$ and the last four are linked to the essential part of the dynamics, the subspace
	$E_3$. It is immediate to see that, for any value of $\alpha \in (0, 2),$ none of these last four eigenvalues is pure imaginary: their square is indeed a complex number with non-zero
	imaginary part, and not a negative real number as it should be. Therefore we conclude that the equilateral triangle is spectrally (hence linearly) unstable for every $\alpha \in (0, 2)$.
	This accords with the fact that every regular polygon is linearly unstable in the gravitational case, as showed by Moeckel in \cite{MR1350320}.
	We also verified (only for $\alpha = 1/2$ and $\alpha = 1$) that the matrix $L_3^{(\alpha)}$ is diagonalisable; unfortunately, due to lack of computational power, we could not check
	if this property is maintained for every other value of the homogeneity parameter in the range of investigation.
	
	As for the logarithmic potential, the angular velocity of the bodies is $\omega = \sqrt{3}$ and the matrix of the linearisation is the following:
		\[
			L^{(\log)} = \left( \begin{smallmatrix}
							0 & -\sqrt{3} & 0 & 0 & 0 & 0 & 1 & 0 & 0 & 0 & 0 & 0 \\
							\sqrt{3} & 0 & 0 & 0 & 0 & 0 & 0 & 1 & 0 & 0 & 0 & 0 \\
							0 & 0 & 0 & -\sqrt{3} & 0 & 0 & 0 & 0 & 1 & 0 & 0 & 0 \\
							0 & 0 & \sqrt{3} & 0 & 0 & 0 & 0 & 0 & 0 & 1 & 0 & 0 \\
							0 & 0 & 0 & 0 & 0 & -\sqrt{3} & 0 & 0 & 0 & 0 & 1 & 0 \\
							0 & 0 & 0 & 0 & \sqrt{3} & 0 & 0 & 0 & 0 & 0 & 0 & 1 \\
							-1 & 0 & \frac{1}{2} & -\frac{\sqrt{3}}{2} & \frac{1}{2} & \frac{\sqrt{3}}{2} & 0 & -\sqrt{3} & 0 & 0 & 0 & 0 \\
							0 & 1 & -\frac{\sqrt{3}}{2} & -\frac{1}{2} & \frac{\sqrt{3}}{2} & -\frac{1}{2} & \sqrt{3} & 0 & 0 & 0 & 0 & 0 \\
							\frac{1}{2} & -\frac{\sqrt{3}}{2} & \frac{1}{2} & \frac{\sqrt{3}}{2} &	-1 & 0 & 0 & 0 & 0 & -\sqrt{3} & 0 & 0 \\
							-\frac{\sqrt{3}}{2} & -\frac{1}{2} & \frac{\sqrt{3}}{2} & -\frac{1}{2} & 0 & 1 & 0 & 0 & \sqrt{3} & 0 & 0 & 0 \\
							\frac{1}{2} & \frac{\sqrt{3}}{2} & -1 & 0 & \frac{1}{2} & -\frac{\sqrt{3}}{2} & 0 & 0 & 0 & 0 & 0 & -\sqrt{3} \\
							\frac{\sqrt{3}}{2} & -\frac{1}{2} & 0 & 1 & -\frac{\sqrt{3}}{2} & -\frac{1}{2} & 0 & 0 & 0 & 0 & \sqrt{3} & 0
						\end{smallmatrix}
					\right).
		\]
	Its eigenvalues are
		\begin{align*}
			\lambda_1 & \= i\sqrt{3}, & \lambda_5 & \= 0, & \lambda_9 & \= i\sqrt{3}, \\
			\lambda_2 & \= -i\sqrt{3}, & \lambda_6 & \= 0, & \lambda_{10} & \= -i\sqrt{3}, \\
			\lambda_3 & \= i\sqrt{3}, & \lambda_7 & \= i\sqrt{6}, & \lambda_{11} & \= i\sqrt{3}, \\
			\lambda_4 & \= -i\sqrt{3}, & \lambda_8 & \= -i\sqrt{6}, & \lambda_{12} & \= -i\sqrt{3}
		\end{align*}
	and as before the last four are connected to the essential subspace $E_3$. Here it is clear that the relative equilibrium is spectrally stable, since every eigenvalue is pure imaginary.
	Nevertheless, it is not linearly stable, because the matrix $L_3^{(\log)}$ is not diagonalisable.
	
	This simple example shows the deep contrast between the $\alpha$-homogeneous potential and the logarithmic one, as well as their similarities: in both cases, indeed, there is linear instability,
	but for opposite reasons.

		\subsection{Linear instability results} \label{subsec:linstab}
	
	We now present a theorem on spectral (hence linear) instability of relative equilibria, valid both in the $\alpha$-homogenous and in the logarithmic case.
	This constitutes an improvement, even in the gravitational case ($\alpha = 1$), of the result found by X.~Hu and S.~Sun in \cite{HuSunCRASP}. Since their proof was only sketched,
	we provide here a complete demonstration and, at the same time, we show that it holds for more general singular potentials.
	In what follows $U$ can be indifferently substituted by $\Uhom$ or $\Ulog$.
	
	Let  $B_3 \in \Mat(4n-8, \R)$ be the restriction of the $4n \times 4n$ matrix $B$ of System~\eqref{eq:hamsyslinearized} to the invariant symplectic subspace $E_3$ of the phase space defined
	by \eqref{eq:E3}. It can be written as
		\[
			B_3 \=
				\begin{pmatrix}
					- \D & \omega \trasp{K}\\
					\omega K & I
				\end{pmatrix},
		\]
	where each block is of dimension $(2n - 4) \times (2n - 4)$ and $\D$ is the restriction of $\trasp{C} D^2 U(\bar{x}) C$ to $E_3$. Following the authors in \cite{HuSunCRASP}, we have:
		\begin{equation} \label{eq:matrixdecomposition}
			\begin{pmatrix}
				I & \omega K\\
				0 & I
			\end{pmatrix}
			\begin{pmatrix}
				- \D & \omega \trasp{K}\\
				\omega K & I
			\end{pmatrix}
			\begin{pmatrix}
				I & 0\\
				-\omega K & I
			\end{pmatrix}=
			\begin{pmatrix}
				-\big(\D + \omega^2 I \big) & 0\\
				0 & I
			\end{pmatrix} \eq N_3.
		\end{equation}
	Note that
		\[
			\D + \omega^2 I \= \trasp{C} D^2 U(\bar{x}) C \bigr|_{E_3} + \omega^2 I = \trasp{C} \big( D^2 U(\bar{x}) + \omega^2 M \big) C \bigr|_{E_3}
		\]
	is precisely the Hessian of $U|_{\mathbb{S}}$ evaluated at the central configuration $\bar{x}$ (cf.~Equations~\eqref{eq:hessristr}, keeping in mind that $\omega^2 = \lambda_\alpha$ if
	$U = \Uhom$ and $\omega^2 = \lambda_{\log}$ if $U = \Ulog$) and restricted to $E_3$.
	
	Define then the nullity and the Morse index of $\bar{x}$ as
		\begin{gather*}
			\nu(\bar{x}) \= \nu \bigl( \D + \omega^2 I \bigr)
		\intertext{and}
			\iMor(\bar{x}) \= \iMor \bigl( \D + \omega^2 I \bigr),
		\end{gather*}
	respectively.
	
	\begin{thm} \label{thm:main}
		Let $\bar{x} \in \mathbb{S}$ be a central configuration for $\Uhom$ or $\Ulog$ such that its nullity $\nu(\bar{x})$ is even.
		If $\iMor(\bar{x})$ is odd, then the corresponding relative equilibrium is spectrally unstable. 
	\end{thm}
	
	\begin{proof}
		Let $\H \= \C^{4n - 8}$ and define the path $D : [0, +\infty) \to \Bsa(\H)$ as
			\[
				D(t) \= B_3 + tG
			\]
		with $G \= iJ$, as above. The proof is then completely analogous to that of Theorem~\ref{thm:mainequiv}, taking into account \eqref{eq:matrixdecomposition} rather than \eqref{eq:NsimB}.
	\end{proof}
	
	\begin{rmk}
		A case occurring quite frequently is $\nu(\bar{x}) = 0$: this happens, for instance, in regular $n$-gons, at least for small values of $n$.
	\end{rmk}
	
	An immediate corollary of this theorem is the following, which, in the gravitational case $\alpha=1$, is the main result of \cite{HuSunCRASP}.
	
	\begin{cor} \label{cor:main}
		Let $\bar{x} \in \mathbb{S}$ be a central configuration for $\Uhom$ or $\Ulog$. If $\iMor(\bar{x})$ or $\nu(\bar{x})$ are odd then the corresponding relative equilibrium is linearly unstable.
	\end{cor}
	
	We shall now derive a useful condition to detect spectral instability of a relative equilibrium utilising only the associated central configuration. Consider again the matrix $L$ of the linearised
	problem given by Equation~\eqref{eq:L}. In the wake of \cite{MR1709850}, we study the eigenvalue problem $Lu = \lambda u$, with $\lambda \in \C$ and
	$u \= \bigl( \begin{smallmatrix} u_1 \\ u_2 \end{smallmatrix} \bigr)$ belonging to $\C^{2n} \times \C^{2n}$ (both $u_1$ and $u_2$ are column vectors):
		\[
			Lu \= \begin{pmatrix}
					\omega K & M^{-1} \\
					D^2 U(\bar{x}) & \omega K
				\end{pmatrix}
				\begin{pmatrix}
					u_1 \\
					u_2
				\end{pmatrix} =
				\begin{pmatrix}
					\lambda u_1 \\
					\lambda u_2
				\end{pmatrix},
		\]
	which corresponds to the system
		\[
			\begin{cases}
				u_2 = M(\lambda I - \omega K)u_1 \\
				Pu_1 = 0,
			\end{cases}
		\]
	where
		\[
			P \= M^{-1} D^2 U(\bar{x}) + (\omega^2 - \lambda^2)I + 2 \lambda \omega K.
		\]
	Thus, to compute the eigenvalues of $L$, it is enough to find those of $P$.
	
	Note that the diagonal $2 \times 2$ blocks of $P$ are of the form
		\[
			\begin{pmatrix}
				d_{ii} + \omega^2 - \lambda^2 & d_{i, i + 1} - 2 \lambda \omega \\
				d_{i + 1, i} + 2 \lambda \omega & d_{i + 1, i + 1} + \omega^2 - \lambda^2
			\end{pmatrix},
		\]
	where the $d_{ij}$'s are the entries of the symmetric matrix $M^{-1}D^2U(\bar{x})$ --- hence $i$ is odd. The determinant of each diagonal block is (setting $\mu \= \lambda^2$)
		\[
			\mu^2 + (2 \omega^2 - d_{ii} - d_{i + 1, i + 1}) \mu + (d_{ii} + \omega^2)(d_{i + 1, i + 1} + \omega^2),
		\]
	so that we have
		\begin{equation} \label{eq:detP}
			\det P = \mu^{2n} + \big( 2n\omega^2 - \tr \big[ M^{-1}D^2 U(\bar{x}) \big] \big) \mu^{2n - 1} + \dotsb,
		\end{equation}
	because the only contribution to the coefficient of $\mu^{2n - 1}$ comes from the diagonal blocks. Now, since the characteristic polynomial of $L$ is even (being $L$ Hamiltonian), from
	Equation~\eqref{eq:detP} we can derive an expression for the sum of the squares of its roots, \ie the eigenvalues $\lambda_i$ of $L$:
		\[
			\sum_{i = 1}^{2n} \mu_i = \sum_{i = 1}^{2n} (\lambda^2)_i = \frac{1}{2} \sum_{i = 1}^{4n} (\lambda_i)^2 = \tr \bigl[ M^{-1} D^2 U(\bar{x}) \bigr] - 2n\omega^2.
		\]
	Recalling the structure of the Hessians of the potentials \eqref{freehessianexpr}, we obtain
		\[
			\tr \big[ M^{-1} D^2 U(\bar{x}) \big] = \begin{cases}
									\displaystyle \sum_{\substack{i, j = 1 \\ i < j}}^{n} \frac{\alpha^2 (m_i + m_j)}{\abs{\bar{x}_i - \bar{x}_j}^{\alpha + 2}} & \text{if $U = \Uhom$} \\ \\
									0 & \text{if $U = \Ulog$}.
								\end{cases}
		\]
	The computation is easily done, noting that $\tr(u_{ij}\trasp{u}_{ij}) = 1$:
		\begin{align}
			\begin{split}
				\tr \big[ M^{-1} D^2 \Uhom(\bar{x}) \big] &= \sum_{i = 1}^n \Bigg\{ - \sum_{\substack{j = 1 \\ j \neq i}}^n
													\frac{\alpha m_j}{\abs{\bar{x}_i - \bar{x}_j}^{\alpha + 2}} \big[ 2 - (\alpha + 2) \big] \Bigg\}\\
											&= \sum_{i = 1}^n \sum_{\substack{j = 1 \\ j \neq i}}^n \frac{\alpha^2 m_j}{\abs{\bar{x}_i - \bar{x}_j}^{\alpha + 2}}
													= \sum_{\substack{i, j = 1 \\ i < j}}^{n} \frac{\alpha^2 (m_i + m_j)}{\abs{\bar{x}_i - \bar{x}_j}^{\alpha + 2}},
			\end{split} \notag \\ & \notag \\
			\tr \big[ M^{-1} D^2 \Ulog(\bar{x}) \big] &= \sum_{i = 1}^n \Bigg\{ - \sum_{\substack{j = 1 \\ j \neq i}}^n \frac{m_j}{\abs{\bar{x}_i - \bar{x}_j}^2} \big[ 2 - 2 \big] \Bigg\} = 0.
				\label{eq:tracelog}
		\end{align}
	
	This discussion proves the following claim.
	
	\begin{thm} \label{thm:sumeigenvalues}
		Let $\bar{z} \= \trasp{(\trasp{\bar{x}}, \bar{y})}$, with $\bar{x} \in \mathbb{S}$ a central configuration, be a relative equilibrium for System~\eqref{eq:hamnuovevariabili} related to $\Uhom$
		(resp.~$\Ulog$), with angular velocity $\omega = \sqrt{\alpha \Uhom(\bar{x})}$ (resp.~$\omega = \sqrt{\mathcal{M}}$), and let $L$ be the matrix \eqref{eq:L} of the associated linearised
		System~\eqref{eq:hamsyslinearized}, with eigenvalues $\lambda_i$ ($i = 1, \dotsc, 4n$). Then we have
			\begin{enumerate}[i)]
				\item $\alpha$-homogeneous case: 
						\begin{equation} \label{eq:sumeigenvaluesuhom}
							\sum_{i = 1}^{4n} (\lambda_i)^2 = 2\alpha^2 \sum_{\substack{i, j = 1 \\ i < j}}^{n} \frac{m_i + m_j}{\abs{\bar{x}_i - \bar{x}_j}^{\alpha + 2}}
																- 4n\alpha\Uhom(\bar{x});
						\end{equation}
				\item Logarithmic case:
						\[
							\sum_{i = 1}^{4n} (\lambda_i)^2 = - 4n\M.
						\]
			\end{enumerate}
	\end{thm}
	
	For a relative equilibrium to be spectrally stable, its eigenvalues must be pure imaginary and therefore their squares must be non-positive. We know the first eight of them, listed in
	Table~\vref{tab:eigenvalues}: the sum of their squares in the $\alpha$-homogeneous case is
		\begin{equation} \label{eq:8eigenvaluesuhom}
			\sum_{i = 1}^{8} (\lambda_i)^2 = (2 \alpha - 8)\omega^2 = 2\alpha (\alpha - 4) \Uhom(\bar{x}).
		\end{equation}
		
	We are now in the position to formulate the following sufficient condition for spectral (hence linear) instability.
	
	\begin{cor} \label{cor:ineqUhom}
		With the hypotheses of Theorem~\ref{thm:sumeigenvalues} (for $U = \Uhom$), if the following inequality holds:
			\begin{equation} \label{eq:ineqUhom}
				\sum_{\substack{i, j = 1 \\ i < j}}^{n} \frac{m_i + m_j}{\abs{\bar{x}_i - \bar{x}_j}^{\alpha + 2}} > \frac{2n + \alpha - 4}{\alpha}\ \Uhom(\bar{x})
			\end{equation}
		then the relative equilibrium $\bar{z}$ is spectrally unstable.
	\end{cor}
	
	\begin{rmk}
		Observe that the relative equilibrium may be degenerate, \ie the matrix $L_3$ may have some zero eigenvalues. We rule out, however, the possibility of complete degeneracy
		($L_3 = 0$): this would correspond indeed to a spectrally stable scenario.
	\end{rmk}
	
	\begin{proof}[Proof of Corollary~\ref{cor:ineqUhom}]
		We prove the contrapositive statement: suppose that the relative equilibrium $\bar{z}$ is spectrally stable. This assumption implies that the sum of the squares of the remaining $4n - 8$
		eigenvalues must be non-positive:
			\[
				\sum_{i = 9}^{4n}(\lambda_i)^2 \leq 0,
			\]
		where equality corresponds to the completely degenerate case where all the eigenvalues of $L_3$ are equal to zero. Adding to both sides the first eight eigenvalues we obtain
			\[
				\sum_{i = 1}^{4n}(\lambda_i)^2 \leq \sum_{i = 1}^{8}(\lambda_i)^2.
			\]
		 Therefore, by Equations~\eqref{eq:sumeigenvaluesuhom} and \eqref{eq:8eigenvaluesuhom}, we get
			\[
				2\alpha^2 \sum_{\substack{i, j = 1 \\ i < j}}^{n} \frac{m_i + m_j}{\abs{\bar{x}_i - \bar{x}_j}^{\alpha + 2}}	- 4n\alpha\Uhom(\bar{x}) \leq 2\alpha (\alpha - 4) \Uhom(\bar{x}).
			\]
		Solving for the summation yields the result.
	\end{proof}
	
	\begin{rmk}
		Note that Corollary~\ref{cor:ineqUhom} provides a tool to detect spectral instability only for the $\alpha$-homogeneous potential $\Uhom$. In the logarithmic case, indeed,
		it is not possible to derive a similar useful condition because of Equation~\eqref{eq:tracelog}. As a justification of this fact, if we let $\alpha \to 0^+$ in \eqref{eq:ineqUhom} we see that the
		left-hand side remains finite, as well as $\Uhom(\bar{x})$, whereas the coefficient on the right-hand side tends to $+\infty$, thus shrinking the solution set of the inequality to
		$\varnothing$. This is not surprising, and is actually in accord with Remark~\ref{rmk:lambdalog}.
	\end{rmk}
	
	As an example of application of Corollary~\ref{cor:ineqUhom}, we examine regular $n$-gons (with $n \geq 3$, as before), employing Roberts' estimates in \cite{MR1709850}. For the
	sake of simplicity, set $m_j \= 1$ for every $j \in \{1, \dotsc, n\}$ and let all the bodies lie at distance $1$ from the origin of the reference frame, positioned at the vertices of a regular
	$n$-gon. In this way $\bar{x}_j = \trasp{\bigl( \cos\frac{2j\pi}{n}, \sin\frac{2j\pi}{n} \bigr)}$ denotes the position of the $j$-th body. Because of the symmetry of this configuration, we have that
	$\abs{\bar{x}_i - \bar{x}_{i + j}} = \abs{\bar{x}_n - \bar{x}_j}$ for all $i, j \in \{1, \dotsc, n\}$, where the indices are understood modulo $n$. Through elementary trigonometry we find
		\[
			\abs{\bar{x}_n - \bar{x}_j} = 2 \sin \frac{j\pi}{n}, \qquad \forall\, j \in \{1, \dotsc, n - 1\},
		\]
	and after a few simplifications Inequality~\eqref{eq:ineqUhom} becomes
		\[
			\sum_{j = 1}^{n - 1} \frac{1}{\sin^{\alpha + 2} \bigl( \frac{j\pi}{n} \bigr)} - \frac{4n + 2\alpha - 8}{n\alpha} \sum_{j = 1}^{n - 1} \frac{1}{\sin^\alpha \bigl( \frac{j\pi}{n} \bigr)} > 0,
		\]
	where we have taken into account the moment of inertia, $\I(\bar{x}) = n$, so that $\bar{x} \in \mathbb{S}$. We make use of the following estimates:
		\[
			\sum_{j = 1}^{n - 1} \frac{1}{\sin^{\alpha + 2} \bigl( \frac{j\pi}{n} \bigr)} \geq \frac{2}{\sin^{\alpha + 2} \bigl( \frac{\pi}{n} \bigr)}, \qquad \qquad
			\sum_{j = 1}^{n - 1} \frac{1}{\sin^\alpha \bigl( \frac{j\pi}{n} \bigr)} \leq \frac{n - 1}{\sin^\alpha \bigl( \frac{\pi}{n} \bigr)}
		\]
	and impose the stronger condition
		\[
			\frac{2}{\sin^{\alpha + 2} \bigl( \frac{\pi}{n} \bigr)} -
										\frac{4n + 2\alpha - 8}{n\alpha} \left[ \frac{n - 1}{\sin^\alpha \bigl( \frac{\pi}{n} \bigr)} \right] > 0.
		\]
	Collecting the common factor $1/\sin^\alpha \bigl( \frac{\pi}{n} \bigr)$, which is positive for every $n \geq 3$, this is equivalent to asking
		\[
			\frac{2}{\sin^2 \bigl( \frac{\pi}{n} \bigr)} - (n - 1)\frac{4n + 2\alpha - 8}{n\alpha} > 0.
		\]
	Exploiting the fact that $\frac{1}{\sin^2 x} > \frac{1}{x^2}$ for every $x \in \R \setminus \pi\Z$, we obtain the solution
		\[
			\bar{\alpha}(n) \= \frac{2\pi^2(n^2 - 3n + 2)}{n^3 - \pi^2n + \pi^2} < \alpha < 2,
		\]
	which is meaningful only for $n \geq 8$. Therefore, for every $n \geq 8$ we see that there exists a real number $\bar{\alpha}(n) \in (0, 2)$ such that for any
	$\alpha \in \bigl( \bar{\alpha}(n), 2 \bigr)$ the regular $n$-gon is spectrally unstable. Moreover, we observe that $\bar{\alpha}(n)$ monotonically tends to $0$ as $n \to +\infty$.

	\appendix
	
	\section{Analytic and symplectic framework} \label{app:A}
	
	The aim of this section is to examine more in depth the analytic and symplectic setting which is used in the rest of the paper, reporting some properties and results supporting and completing
	the previous propositions. The notation is the same as that one adopted in Subsection~\ref{subsec:spfl}.

		\subsection{On the spectral flow}
	
	We present here some important properties of the spectral flow. Our basic reference is \cite{Les05}.
	
	\begin{thm}[\cite{Les05}] \label{thm:conc-hominv}
		Let
			\[
				\mu : \Omega \bigl( \Bsa(\H), \G\Bsa(\H) \bigr) \to \Z,
			\]
		be a map which satisfies the following properties:
			\begin{enumerate}[i)]
				\item \emph{Concatenation:} If $\gamma, \delta \in \Omega \bigl( \Bsa(\H), \G\Bsa(\H) \bigr)$, with $\gamma(b) = \delta(a)$, then
						\[
							\mu(\gamma * \delta) = \mu(\gamma) + \mu(\delta).
						\]
				\item \emph{Homotopy invariance:} The map $\mu$ descends to a map $\tilde{\mu} : \tilde{\pi}_1 \bigl( \Bsa(\H), \G\Bsa(\H) \bigr) \to \Z$, that is, the following diagram is
					commutative ($p$ denotes the quotient map):
						\[
							\xymatrix{
								\Omega \bigl( \Bsa(\H), \G\Bsa(\H) \bigr) \ar[r]^-{\mu} \ar[d]_-{p} & \Z \\
								\tilde{\pi}_1\bigl( \Bsa(\H), \G\Bsa(\H) \bigr) \ar[ur]_-{\tilde{\mu}}.
								}
						\]
				\item \emph{Normalisation:} There exist an orthogonal projector $P \in \Bsa(\H)$ of rank $1$ such that
						\begin{enumerate}[a)]
							\item the restriction $(I - P)A(I - P)|_{\ker P}$ of the operator $(I - P)A(I - P) \in \Bsa(\H)$ to the kernel of $P$ is invertible for every $A \in \Bsa(\H)$;
							\item the path $\zeta \in \Omega \bigl( \Bsa(\H), \G\Bsa(\H) \bigr)$ defined by
									\[
										\zeta(t) \= \biggl( t - \frac{1}{2} \biggr)P + (I - P)A(I - P) 	\qquad \textup{for all } t \in [0, 1]
									\]
								verifies
									\[
										\mu (\zeta) = 1.
									\]
						\end{enumerate}
			\end{enumerate}
		Then 
			\[
				\mu(\gamma) = \spfl \bigl( \gamma, [a, b] \bigr)
			\]
		for all $\gamma \in \Omega \bigl( \Bsa(\H), \G\Bsa(\H) \bigr)$.
	\end{thm}
	
	\begin{rmk}
		If we fix a basis $(e_1, \dotsc, e_n)$ in $\H$, then the axiom of normalisation in the previous theorem can be stated as follows. Let $P \in \Bsa(\H)$ be an orthogonal projector
		whose image is generated by $e_1$ and for a fixed $k \in \{2, \dotsc n - 1\}$ define two other orthogonal projectors $P_k^+$ and $P_k^-$ by $\im P_k^+ \= \gen\{ e_2, \dotsc, e_k \}$ and
		$\im P_k^- \= \gen\{ e_{k + 1}, \dotsc, e_n \}$. Choose $A \= P_k^+ - P_k^-$. Then the path $\zeta \in \Omega \bigl( \Bsa(\H), \G\Bsa(\H) \bigr)$ given by
			\[
				\zeta(t) \= \biggl( t - \frac{1}{2} \biggr)P + A	\qquad \textup{for all } t \in [0, 1]
			\]
		satisfies $\mu(\zeta) = 1$.
		
		This is actually a particular case of what we wrote in Theorem~\ref{thm:conc-hominv}, but we observe that it can be used as well to declare which paths have spectral flow equal to $1$.
		
		We note that our formulation of this axiom corrects the statement of \cite[Theorem~5.7]{Les05}, in which there is clearly just an oversight: the condition of invertibility of $(I - P)A(I - P)$ is
		indeed missing there.
	\end{rmk}
	
	\begin{lem}\label{thm:lemmafava}
		Let $t_* \in \R$ and consider a path $T \in \Cscr^1\bigl( [t_* - \eps, t_* + \eps], \Bsa(\H) \bigr)$, for some $\eps > 0$. Suppose that $T$ has a unique
		regular crossing at $t = t_*$. Then
			\[
	       			\spfl \bigl( T, [t_* - \eps, t_* + \eps] \bigr) = \sgn \Gamma (T, t_*).
			\]
	\end{lem}
	
	\begin{proof}
		Let $Q : \H \to \H$ be the orthogonal projection onto the kernel of $T(t_*)$. Since $t_*$ is a regular crossing instant for $T$, the operator $Q \dot{T}(t_*) Q|_{\H_{t_*}}$ is invertible on
		$\H_{t_*} \= \ker T(t_*)$. Therefore there exists a number $\beta > 0$ such that $Q\bigl( \dot{T}(t_*) + B \bigr)Q|_{\H_{t_*}}$ is also invertible on $\H_{t_*}$ for every $B \in \Bsa(\H)$
		such that $\norm{B} < \beta$. On the other hand, being $T(t_*)\rvert_{\H_{t_*}} = 0$, we may choose a number $\eps > 0$ such that
			\[
				\norm{\biggl( \frac{T(t) - T(t_*)}{t - t_*} - \dot{T}(t_*) \biggr) \biggr\rvert_{\H_{t_*}}} = \norm{\biggl( \frac{T(t)}{t - t_*} - \dot{T}(t_*) \biggr) \biggr\rvert_{\H_{t_*}}} < \beta
			\]
		for every $t \in [t_* + \eps, t_* + \eps] \setminus \{ t_* \}$.
		
		Define then a homotopy $F : [0,1] \times [t_* - \eps, t_* + \eps] \to \Bsa(\H_{t_*})$ by
			\[
				\begin{split}
					F(s, t) & \= s T(t) \bigr|_{\H_{t_*}} + (1 - s)(t - t_*) \dot{T}(t_*) \bigr|_{\H_{t_*}} \\
						& \,= (t - t_*) \biggl[ s \biggl( \frac{T(t)}{t - t_*} - \dot{T}(t_*) \biggr) \biggr\rvert_{\H_{t_*}} + \dot{T}(t_*) \bigr|_{\H_{t_*}} \biggr].
				\end{split}
			\]
		The previous choice of $\eps$ is thus sufficient to guarantee that $F(s, t)$ is invertible for every $s \in [0,1]$ and every $t \in [t_* - \eps, t_* + \eps]$. Hence, by the homotopy invariance
		of the spectral flow,
			\begin{equation} \label{eq:hidden}
				\spfl \bigl( T, [t_* - \eps, t_* + \eps] \bigr) = \spfl \bigl( (t - t_*) \dot{T}(t_*) \bigr|_{\H_{t_*}}, [t_* - \eps, t_* + \eps] \bigr).
			\end{equation}
		Here we actually use the fact that, for all $t \in [t_* - \eps, t_* + \eps]$, the operator $T(t)$ splits into $T|_{\H_{t_*}}(t) + T|_{\H_{t_*}^\perp}(t)$ on $\H = \H_{t_*} \oplus \H_{t_*}^\perp$
		and that the spectral flow is compatible with this splitting, \ie
			\[
				\spfl \bigl( T \bigr|_{\H_{t_*}} + T \bigr|_{\H_{t_*}^\perp}, [t_* - \eps, t_* + \eps] \bigr) =
				\spfl \bigl( T \bigr|_{\H_{t_*}}, [t_* - \eps, t_* + \eps] \bigr) + \spfl \bigl( T \bigr|_{\H_{t_*}^\perp}, [t_* - \eps, t_* + \eps] \bigr).
			\]
		The last addendum is of course zero and this justifies equality~\eqref{eq:hidden}.
		
		Finally, by Remark~\ref{rmk:spfl},
			\[
				\begin{split}
				\spfl \bigl( (t - t_*) \dot{T}(t_*) \bigr|_{\H_{t_*}}, [t_* - \eps, t_* + \eps] \bigr)
				& = n^- \bigl( -\eps \dot{T}(t_*) \bigr|_{\H_{t_*}} \bigr) - n^- \bigl( \eps\dot{T}(t_*) \bigr|_{\H_{t_*}} \bigr) \\
				& = n^+ \bigl( \dot{T}(t_*) \bigr|_{\H_{t_*}} \bigr) - n^- \bigl( \dot{T}(t_*) \bigr|_{\H_{t_*}} \bigr) \\
				& = \sgn \dot{T}(t_*) \bigr|_{\H_{t_*}} \\
				& = \sgn Q \dot{T}(t_*) Q\rvert_{\H_{t_*}}. \qedhere
				\end{split} 
			\]
	\end{proof}
	
	From this Lemma immediately follows the next Proposition.
	
	\begin{prop}\label{thm:sfcasoregolare}
		Let $T \in \Cscr^1\bigl( [0,1], \Bsa(\H) \bigr)$ be a regular curve with invertible endpoints. Then the spectral flow is computed as:
			\begin{equation} \label{eq:sfcasoregolare}
				\spfl \bigl( T, [0, 1] \bigr) =\ \sum_{\mathclap{\substack{t_* \in [0,1] \\ t_* \textup{ crossing}}}}\ \sgn \Gamma (T, t_*).
			\end{equation}
	\end{prop}
		
	\begin{proof}
		Since every crossing is regular by assumption, the corresponding crossing forms are all non-degenerate and we can use the Inverse Function Theorem to deduce that the
		crossings are isolated. Then we can apply Lemma~\ref{thm:lemmafava} to each isolated crossing and sum up every contribution by means of the concatenation property of the
		spectral flow. The compactness of the interval $[0,1]$ ensures that there are only finitely many crossing and that the sum on the right-hand side of \eqref{eq:sfcasoregolare} is
		well defined.
	 \end{proof}
	 	  
	In the following proposition we investigate the parity of the spectral flow of an affine path of Hermitian matrices. 

	\begin{prop}\label{prop:MainA}
		Let $\H$ be a complex Hilbert space of dimension $4n$, let $A \in \Bsa(\H) \setminus \G\Bsa(\H)$ be a real symmetric non-invertible matrix and take $C \in \G\Bsa(\H)$ of the form
		$C \= iB$, where $B$ is a real skew-symmetric invertible matrix. Consider the affine path $D : [0, +\infty) \to \Bsa(\H)$ defined by
			\[
				D(t) \= A + tC.
			\]
		Let $E_\lambda \= \ker (B^{-1}A + \lambda I)$ be the eigenspace of $-B^{-1}A$ relative to the eigenvalue $\lambda$ and let $Q_\lambda : \H \to \H$ be the eigenprojection onto
		$E_\lambda$. We assume that
			\begin{enumerate}[(H1)]
				\item The quadratic form $Q_\lambda C Q_\lambda \bigr|_{E_\lambda}$ is non-degenerate for every $\lambda \in \sigma(-B^{-1}A)$;
				\item $-B^{-1}A$ is diagonalisable;
				\item $\sigma(-B^{-1}A) \subset i\R$ and it is symmetric with respect to the real axis;
				\item $\nu(A)$ is even.
	 		\end{enumerate}
	    	Then there exist $\eps > 0$ and  $T > \eps$ such that 
			\begin{enumerate}[(T1)]
				\item The instant $t = 0$ is the only crossing for $D$ on $[0, \eps]$;
				\item $\spfl \bigl( D, [\eps, T_1] \bigr) = \spfl \bigl( D, [\eps, T_2] \bigr)$ for all $T_1, T_2 \geq T$;
				\item $\spfl \bigl( D, [\eps, T] \bigr)$ is even.
			\end{enumerate}
	\end{prop}
	
	\begin{proof}
		Statement (T1) follows by assumption (H1). Indeed, $t = 0$ is a crossing instant because $A$ is singular and it is regular because the crossing form
		$Q_0 C Q_0|_{E_0}$ is non-degenerate. Thus, by the Inverse Function Theorem, it is isolated and the number $\eps > 0$ claimed in the first thesis exists.
	
		In order to prove (T2), we observe that there exist $T > \eps$ such that
			\[
				\sgn D(t) = \sgn C, \qquad \forall\, t \geq T.
			\]
		To prove this claim, we analyse the following two cases (note that $\sigma(C) \subset \R \setminus \{0\}$, being $C$ hermitian and invertible):
			\begin{itemize}
				\item $\lambda_* \in \sigma(C) \cap \R^-$. If $u_* \in \ker(C - \lambda_* I)$ is an eigenvector related to $\lambda_*$, we have
						\[
							\langle D(t) u_*, u_* \rangle = \langle A u_*, u_*\rangle + t \lambda_* \norm{u_*}^2.
	     					\]
					Thus
						\[
							\sup_{\mathclap{\substack{\norm{u} = 1\\ u\, \in\, \ker(C - \lambda_*I)}}}\ \langle D(t) u, u \rangle \leq \lambda_{\textup{max}} + t \lambda_*,
						\]
					where $\lambda_{\textup{max}}$ is the maximum of the quadratic form $\langle Au,u \rangle$ on the unit sphere of the eigenspace of $C$ relative to $\lambda_*$
					(which is attained by Weierstra\ss\ theorem). If we choose $T_{\textup{max}} \= \dfrac{1 + \lambda_{\textup{max}}}{\abs{\lambda_*}}$, we obtain that
						\[
							\sup_{\mathclap{\substack{\norm{u} = 1\\ u\, \in\, \ker(C - \lambda_*I)}}}\ \langle D(t) u, u \rangle \leq -1, \qquad \forall t \geq T_{\textup{max}},
						\]
					so that $\lambda_*$ eventually defines a negative eigendirection for $D(t)$.

	   			\item $\lambda_* \in \sigma(C) \cap \R^+$. If $u_* \in \ker(C - \lambda_*I)$ is an eigenvector related to $\lambda_*$, we have
						\[
							\langle D(t) u_*, u_* \rangle = \langle A u_*, u_*\rangle + t \lambda_* \norm{u_*}^2.
						\]
					Thus
						\[
							\sup_{\mathclap{\substack{\norm{u} = 1\\ u\, \in\, \ker(C - \lambda_*I)}}}\  \langle D(t) u, u \rangle \geq \lambda_{\textup{min}} + t \lambda_*,
						\]
					where $\lambda_{\textup{min}}$ is the minimum of the quadratic form $\langle Au, u \rangle$ on the unit sphere of the eigenspace of $C$ relative to $\lambda_*$
					(which is attained by Weierstra\ss\ theorem). If we choose $T_{\textup{min}} \= \dfrac{1 - \lambda_{\textup{min}}}{\abs{\lambda_*}}$, we obtain that
						\[
							\sup_{\mathclap{\substack{\norm{u} = 1\\ u\, \in\, \ker(C - \lambda_*I)}}}\ \langle D(t) u, u \rangle \geq 1, \qquad \forall t \geq T_{\textup{min}},
						\]
					so that $\lambda_*$ eventually defines a positive eigendirection for $D(t)$.
	 		\end{itemize}
		Define then $T \= \max\{ T_{\textup{min}}, T_{\textup{max}} \}$. Without loss of generality we may assume that $T_1 < T_2$. By means of the concatenation property of the spectral flow
		we get
			\[
				\spfl \bigl( D, [\eps, T_2] \bigr) = \spfl \bigl( D, [\eps, T_1] \bigr) + \spfl \bigl( D, [T_1, T_2] \bigr) = \spfl \bigl( D, [\eps, T_1] \bigr),
			\]
		where the last equality comes from the fact that $D(t)$ is an isomorphism for every $t \geq T$.
		As we showed, indeed, for any $t \geq T$ each eigenvalue of $C$ determines an eigendirection (and hence an eigenvalue) of $D(t)$ of the same sign. Being $C$ invertible, the claim
		follows.

		We now prove (T3). Let us first make the link between $t_*$ and $\lambda$ explicit: writing
			\[
				D(t) = -B(-B^{-1}A - itI)	\qquad \forall t \in [0, +\infty)
			\]
		it is clear that $t_*$ is a crossing for $D$ if and only if $\lambda = it_*$ is an eigenvalue of $-B^{-1}A$. Now, by construction both $\eps$ and $T$ are not crossing instants for $D$, hence
		we can apply Proposition~\ref{thm:sfcasoregolare} to $\spfl \bigl( D, [\eps, T] \bigr)$ and write
			\begin{equation}\label{eq:proof1}
				\spfl \bigl( D, [\eps, T] \bigr) = \sum_{\mathclap{\substack{t_*\, \in\, [\eps, T] \\ t_* \textup{ crossing}}}}\ \sgn \Bigl( Q_\lambda C Q_\lambda \bigr|_{E_\lambda} \Bigr).
			\end{equation}
		(Note that this summation is meaningful because of our brief discussion a few lines above.)
		Since the crossing forms are non-degenerate by (H1) and since $-B^{-1}A$ is diagonalisable by (H2), we have that
			\begin{equation} \label{eq:proof2}
				\begin{split}
					\sgn \Bigl( Q_\lambda C Q_\lambda \bigr|_{E_\lambda} \Bigr) & \=
								n^+\Bigl( Q_\lambda C Q_\lambda \bigr|_{E_\lambda} \Bigr) - n^-\Bigl( Q_\lambda C Q_\lambda \bigr|_{E_\lambda} \Bigr) \\
								& \equiv n^+\Bigl( Q_\lambda C Q_\lambda \bigr|_{E_\lambda} \Bigr) + n^-\Bigl( Q_\lambda C Q_\lambda \bigr|_{E_\lambda} \Bigr) \qquad \mod 2 \\
								& = \dim E_\lambda
				\end{split}
			\end{equation}
		for all $\lambda \in \sigma(-B^{-1}A)$. As a consequence of \eqref{eq:proof1} and \eqref{eq:proof2} we infer that
			\begin{equation}\label{eq:utile}
				\spfl \bigl( D, [\eps, T] \bigr)\ \equiv\ \sum_{\mathclap{\substack{t_*\, \in\, [\eps, T] \\ t_* \textup{ crossing}}}}\ \dim E_\lambda
										=\qquad \sum_{\mathclap{\lambda\, \in\, \sigma(-B^{-1}A)\, \cap\, i [\eps, +\infty)}}\ \dim E_\lambda \quad 	\mod 2.
			\end{equation}
		Finally, taking into account (H2), (H3) and (H4), we deduce
			\[
				\sum_{\mathclap{\lambda\, \in\, \sigma(-B^{-1}A)\, \cap\, i [\eps, +\infty)}}\ \dim E_\lambda = 2n - \nu(A) \equiv 0 \quad \mod 2. \qedhere
			\]
	\end{proof}
	
	The next corollary is directly derived from the previous proposition and it deals with the case when the matrix $A$ is invertible.
	
	\begin{cor} \label{cor:MainA}
		In the same setting of Proposition~\ref{prop:MainA}, assume that $A \in \G\Bsa(\H)$ and that (H1), (H2) and (H3) hold.
		Then there exists $T > 0$ such that 
			\begin{enumerate}
				\item[(T2')] $\spfl \bigl( D, [0, T_1] \bigr) = \spfl \bigl( D, [0, T_2] \bigr)$ for all $T_1, T_2 \geq T$;
				\item[(T3')] $\spfl \bigl( D, [0, T] \bigr)$ is even.
			\end{enumerate}
	\end{cor}
	
	\begin{proof}
		Since $A$ is invertible, the instant $t = 0$ is not a crossing for $L$ and $\nu(A) = 0$. Consequently, we can compute the spectral flow of $L$ directly on the interval $[0, T]$ and apply
		Proposition~\ref{prop:MainA} with obvious modifications.
	\end{proof}

		\subsection{Root functions, partial signatures and spectral flow} \label{subsec:parsgn}
	
	The aim of this section is to derive a formula for computing the spectral flow of an affine path at a possibly degenerate (\ie non-regular) crossing instant.
	The main references are \cite{MR2057171, arxiv} and references therein.
	
	Let $t_* \in \R$, $\eps > 0$ and $T : [t_* - \eps, t_* + \eps] \to \Bsa(\H)$ be a real-analytic path such that $t = t_*$ is an isolated crossing for $T$. We are interested in computing the
	\lq\lq jumps\rq\rq\ of the functions $n^+ \bigl( T(t) \bigr)$ and $n^- \bigl( T(t) \bigr)$ as $t$ passes through $t_*$ also in the degenerate case, in order to generalise Lemma~\ref{thm:lemmafava}.
	Recall that, given $k \in \N \setminus \{0\}$, a smooth map $f : [t_* - \eps, t_* + \eps] \to \H$ is said to have a \emph{zero of order $k$} at $t = t_*$ if
	$f(t_*) = f'(t_*) = \ldots = f^{(k - 1)}(t_*) = 0$ and $f^{(k)}(t_*) \neq 0$. We recall that in this case both the eigenvalues and the eigenvectors of $T(t)$ are real-analytic functions defined on the
	domain of $T$ (see \cite[Chapter~2]{kato}); we denote them by $\lambda_i(t)$ and $v_i(t)$ respectively, for $i \in \{ 1, \dotsc, \dim\H \}$. Moreover, if $\lambda_i(t)$ vanishes at $t = t_*$ for
	some index $i$ then it has a zero of finite order, and for each $i$ the $v_i$'s are pairwise orthogonal unit eigenvectors relative to the $\lambda_i(t)$'s.
	
	\begin{defn} \label{thm:defroot}
		A \emph{root function} for $T(t)$ at $t = t_*$ is a smooth map $u : [t_* - \eps, t_* + \eps] \to \H$ such that $u(t_*) \in \ker T(t_*)$.
		The \emph{order} $\ord(u)$ of the root function $u$ is the (possibly infinite) order of the zero at $t = t_*$ of the map $t \mapsto T(t) u(t)$.
	\end{defn}
	
	\noindent
	In correspondence of the (possibly non-regular) crossing instant $t_*$ for $T$ we define, for every $k \in \N \setminus \{ 0 \}$, a descending filtration $(\W_k)$ of vector spaces
	$\W_k \subset \H$ and a sequence $(\mathcal{B}_k)$ of sesquilinear forms $\mathcal{B}_{k} : \W_{k}\times \W_{k} \to \C$ as follows:
		\begin{gather}
			\W_k \= \Set{u_* \in \H | \exists \textup{ a root function $u$ with $\ord(u) \geq k$ and $u(t_*) = u_*$} }, \notag \\
			\mathcal{B}_k(u_*, v_*) \=  \frac{1}{k!} \biggl\langle \frac{\d^k}{\d t^k} \bigl[ T(t)u(t) \bigr] \Bigr|_{t\, =\, t_*}, v_* \biggr\rangle \qquad \forall\, u_*, v_* \in \W_k, \label{eq:Bk}
		\end{gather}
	where $u$ in \eqref{eq:Bk} is any root function with $\ord(u) \geq k$ and $u(t_*) = u_*$. The right-hand side of the equality in \eqref{eq:Bk} is well defined and indeed it turns out to be
	independent of the choice of the root function $u$ (see \cite[Proposition 2.4]{arxiv}).
	
	\begin{defn} \label{thm:defpartsign}
		For all $k \in \N \setminus \{ 0 \}$, the integer number
			\[
				\sgn_k(T, t_*) \= \sgn \mathcal{B}_k
			\]
		is called the \emph{$k$-th partial signature} of $T(t)$ at $t = t_*$.
	\end{defn}

	\begin{prop} \label{thm:central}
		Let $t_* \in \R$, $\eps > 0$ and $T : [t_* - \eps, t_* + \eps] \to \Bsa(\H)$ be a real-analytic path having a unique (possibly non-regular) crossing at $t = t_*$.
		Then
			\begin{enumerate}[(i)]
				\item $\W_k = \gen \Set{ v_i(t_*) \in \H | \lambda_i^{(j)}(t_*) = 0 \textup{ for all } j < k \textup{ and } \lambda_i^{(k)}(t_*) \neq 0 }$;
				\item If $v \in \W_k$ is an eigenvector of $\lambda(t_*)$ then $\mathcal{B}_k(v, w) = \frac{1}{k!}\lambda^{(k)}(t_*)\langle v, w \rangle$, for all $w \in \W_k$;
				\item $\displaystyle \spfl \bigl( T, [t_* - \eps, t_* + \eps] \bigr) = \sum_{k = 1}^{+\infty}\sgn_{2k-1}(T, t_*)$, where the sum has only finitely many non-zero terms.
			\end{enumerate}
	\end{prop}

	\begin{proof}
		It follows verbatim from \cite[Proposition~2.9, Corollary~2.14]{arxiv}: the results there contained hold also if the underlying Hilbert space is complex.
	\end{proof}
	
	\begin{rmk}
		Part (iii) of Proposition~\ref{thm:central} is the generalisation of Lemma~\ref{thm:lemmafava} to the degenerate case that we were seeking. 
	\end{rmk}
		
	We close this subsection with the following central result, which computes the spectral flow for a path of Hermitian matrices in terms of partial signatures.
	
	\begin{prop}\label{thm:cruciale}
		Let $A \in \Bsa(\H)$ and $C \in \G\Bsa(\H)$. Consider the affine path $\tilde{D} : (0, +\infty) \to \Bsa(\H)$ defined by
			\[
				\tilde{D}(s) \= sA + C
			\]
		and assume that $s_* \in (0, +\infty)$ is an isolated (possibly non-regular) crossing instant for $\tilde{D}$, so that $1/s_*$ is an eigenvalue of $-C^{-1}A$.
		Then for $\eps > 0$ small enough
			\[
				\spfl \bigl( \tilde{D}, [s_* - \eps, s_* + \eps] \bigr) = -\sgn \mathcal{B}_1,
			\]
		where
			\[
				\mathcal{B}_1 \= \langle C\,\cdot, \cdot \rangle \bigr|_{\H_{s_*}}
			\]
		and $\H_{s_*}$ is the generalised eigenspace
			\[
				\H_{s_*} \=\bigcup_{j = 1}^{\dim \H} \ker \biggl( C^{-1}A + \dfrac{1}{s_*} I \biggr)^j.
			\]
	\end{prop}
	
	\begin{proof}
		See \cite[Corollary~3.30]{arxiv}.
	\end{proof}

		\subsection{Krein signature of a complex symplectic matrix} \label{subsec:Krein}
	  
	We now briefly recall some basic facts about the Krein signature of a symplectic matrix. Our main references are the books \cite[Chapter 1]{Abbo} and \cite{long}.
	  
	Let $S \in \Sp(2n, \R)$ be a real symplectic matrix. In order to define the \emph{Krein signature} of the eigenvalues of $S$, we consider the usual action of $S$ on $\C^{2n}$
		\[
			S(\xi + i\eta) \= S \xi + i S \eta, \qquad \forall\, \xi, \eta \in \R^{2n},
		\]
	and the Hermitian form $g : \C^{2n} \times \C^{2n} \to \R$ given by
		\[
			g(v, w) \= \langle Gv, w \rangle \qquad \forall\, v, w \in \C^{2n},
		\]
	where $\langle \cdot, \cdot \rangle$ denotes the standard scalar product in $\C^{2n}$. The \emph{complex symplectic group} $\Sp(2n, \C)$ is the set of all complex linear automorphisms of
	$\C^{2n}$ which preserve $g$ or, equivalently, the set of all complex matrices $S$ satisfying the condition $S^\dagger JS = J$. A matrix is an element of $\Sp(2n, \R)$ if and only if it belongs
	to $\Sp(2n, \C)$ and it is real. Following the discussion in \cite[pages~12--13]{Abbo} and \cite[Chapter~1]{long}, it is possible to show that the spectral decomposition of $\C^{2n}$
		\[
			\C^{2n} = \bigoplus_{\mathclap{\substack{\lambda\, \in\, \sigma(S)\\ \abs{\lambda}\, \geq\, 1}}}\ F_\lambda,
		\]
	where
		\[
			F_\lambda \= \begin{cases}
						E_\lambda & \textup{if } \abs{\lambda} = 1 \\
						E_\lambda \oplus E_{\overline{\lambda}^{-1}} & \textup{if } \abs{\lambda} > 1
					\end{cases}
		\]
	and
		\[
			E_\lambda \= \bigcup_{j = 1}^{2n} \ker (S - \lambda I)^j,
		\]
	is $g$-orthogonal. Therefore each restriction $g|_{F_\lambda}$ is non-degenerate for all $\lambda \in \sigma(S)$.
	 
	\begin{rmk} \label{sgndim}
		Because of the non-degeneracy of $g$ on each space $F_\lambda$ we obtain that
			\[
				\begin{split}
					\sgn\, g \bigr|_{F_\lambda} & \= n^+\bigl( g \big|_{F_\lambda} \bigr) - n^-\bigl( g \big|_{F_\lambda} \bigr) \\
										 & \equiv n^+\bigl( g \big|_{F_\lambda} \bigr) + n^-\bigl( g \big|_{F_\lambda} \bigr) \qquad \mod 2 \\
										 & = \dim F_\lambda. 
				\end{split}
			\]
	\end{rmk}
	
	If $\lambda \in \sigma(S) \setminus \U$ has algebraic multiplicity $d$, then $g$ restricted to the $2d$-dimensional subspace $E_\lambda \oplus E_{\overline{\lambda}^{-1}}$ has a
	$d$-dimensional isotropic subspace. Thus $g$ has zero signature on $E_\lambda \oplus E_{\overline{\lambda}^{-1}}$. On the contrary, an eigenvalue $\lambda \in \sigma(S) \cap \U$ may
	have any signature on $E_\lambda$, and therefore we are entitled to give the following definition.
	
	\begin{defn}\label{def:Kreinsignature}
		Let $S \in \Sp(2n, \C)$ be a complex symplectic matrix and let $\lambda \in \sigma(S) \cap \U$ be a unitary eigenvalue of $S$.
		The \emph{Krein signature} of $\lambda$ is the signature of the restriction $g|_{E_\lambda}$ of the Hermitian form $g$ to the generalised eigenspace $E_\lambda$.
	\end{defn}
	
	Assume that $S \in \Sp(2n,\R)$. If an eigenvalue $\lambda \in \sigma(S) \cap \U$ has Krein signature $p$, then its complex conjugate $\overline{\lambda}$ (which is again an eigenvalue of
	$S$ because of the properties of the spectrum of symplectic matrices, cf.~Proposition~\ref{prop:spectra}) has Krein signature $-p$. This implies, in particular, that $1$ and $-1$ always have
	Krein signature $0$\label{page:sign0}.


	
	\addcontentsline{toc}{section}{\refname}	
	\bibliographystyle{amsplain}			
	\bibliography{BarJadPor13_ref}			

	\vfill
	
	\vspace{1cm}
	\noindent
	\textsc{Vivina L.~Barutello}\\
	Dipartimento di Matematica \lq\lq G.~Peano\rq\rq\\
	Universit\`a degli Studi di Torino\\
	Via Carlo Alberto, 10 \\
	10123 Torino \\
	Italy\\
	E-mail: \email{vivina.barutello@unito.it}

	\vspace{1cm}
	\noindent
	\textsc{Riccardo Danilo Jadanza}\\
	Dipartimento di Scienze Matematiche \lq\lq J.-L.~Lagrange\rq\rq\ (DISMA)\\
	Politecnico di Torino\\
	Corso Duca degli Abruzzi, 24\\
	10129 Torino\\
	Italy\\
	E-mail: \email{riccardo.jadanza@polito.it}

	\vspace{1cm}
	\noindent
	\textsc{Alessandro Portaluri}\\
	Dipartimento di Scienze Agrarie, Forestali e Alimentari (DISAFA)\\
	Universit\`a degli Studi di Torino\\
	Via Leonardo da Vinci, 44\\
	10095 Grugliasco (TO)\\
	Italy\\
	E-mail: \email{alessandro.portaluri@unito.it}\\
	Website: \href{http://aportaluri.wordpress.com}{\textsf{http://aportaluri.wordpress.com}}

\end{document}